\documentclass[11pt,a4paper,leqno]{amsart}
\usepackage{amsmath,amsfonts}
\usepackage{amssymb}
\usepackage{upref}
\usepackage[mathcal]{eucal}

\date{\today}

\begin{document}

\title{Partitions of primes by  Chebyshev polynomials}
\author{Maciej P. Wojtkowski}
\address{University of Opole\\
 Opole, Poland}
\email{mwojtkowski@uni.opole.pl}
\date{\today}

            \theoremstyle{plain}
\newtheorem{lemma}{Lemma}%[section]
\newtheorem{proposition}[lemma]{Proposition}
\newtheorem{theorem}[lemma]{Theorem}
\newtheorem{corollary}[lemma]{Corollary}
\newtheorem{fact}[lemma]{Fact}
   \theoremstyle{definition}
\newtheorem{definition}{Definition}
    \theoremstyle{example}
\newtheorem{example}{Example}
    \theoremstyle{remark}
\newtheorem{remark}{Remark}

\newcommand{\macierz}[4]{\scriptsize{\left[\begin{array}{cc} #1 & #2 \\
        #3 & #4 \end{array}\right]}}

\begin{abstract}
  Partitions of the set of primes are introduced based on the
  Chebyshev polynomials at  rationals. The prime densities of
  all  such partitions are established.

  Euler's Criterion for $SL(2,\mathbb Q)$ is formulated, which is
  the bridge between the algebra of Chebyshev polynomials
  and number-theoretic properties of the partitions.
  
  It is shown how to obtain in this way some
   of the  classical theory of Lucas  sequences.  
A hidden symmetry of the problem is revealed by the new language.
   
As an application  
number-theoretic properties of simple dynamical
systems (rotations  and certain interval maps) are discussed.
\end{abstract}

\maketitle

\section{Introduction}
For every $2\times 2$ matrix $A$ with determinant $1$
we have $A^2 = qA-I$, where $q$ is the trace of $A$,
and $I$ is the identity matrix.
It follows  that for any integer $n$ the power $A^n$ is equal to
\begin{equation}
  A^n = U_{n}(q)A-U_{n-1}(q)I,
  \end{equation}
for some polynomials $U_n(q), n\in \mathbb Z$, of degree $|n|-1$,
with integer coefficients,
which are called the {\it Chebyshev polynomials   of the second kind}.
The traces of the powers $A^n$ are also polynomials in $q$
with integer coefficients, $C_n(q) = U_{n+1}(q)-U_{n-1}(q)$,
which are called the {\it Chebyshev polynomials   of the first kind}.
The list of the first fourteen of them is provided in Appendix B.

The polynomials $C_n$ and $U_n$ satisfy the same recursive relations,
which can be put concisely into the following matrix identities.  
\begin{equation}
  \begin{aligned}
    &      \left[ \begin{array}{cc} -U_{n-1} &     U_{n} \\
          -U_{n}  &   U_{n+1} \end{array} \right]=
    \left[ \begin{array}{cc} 0 &      1 \\
          -1  &   q \end{array} \right]^n \\
    &\left[ \begin{array}{cc} C_{n} &       U_{n} \\
      C_{n+1}  &   U_{n+1} \end{array} \right]=
      \left[ \begin{array}{cc} 0 &      1 \\
          -1  &   q \end{array} \right]^n
            \left[ \begin{array}{cc} 2 &       0 \\
                q  &   1 \end{array} \right].
            \end{aligned}
\end{equation}
Note that in most applications the Chebyshev polynomials are chosen
to be the functions of half the trace,  \cite{Wiki}.
Our choice is compelling for  number theoretical questions.
The role of Chebyshev
polynomials for the group $GL(2,\mathbb C)$ was explored
by Damphousse, \cite{D}, and it was rediscovered in \cite{A}.

We consider two more sequences of polynomials
\[
V_{2k+1} =U_{k+1}-U_{k}, \ \ \ W_{2k+1} =U_{k+1}+U_{k}, \ \ k \in \mathbb Z.
\]
These  polynomials 
are sometimes called the {\it Chebyshev polynomials of the third} and
{\it fourth kind},
respectively, \cite{Y}.

They satisfy the same recursive relations as Chebyshev polynomials. 
\begin{equation}
      \left[ \begin{array}{cc} V_{2k-1} &       W_{2k-1} \\
      V_{2k+1}  &   W_{2k+1} \end{array} \right]=
      \left[ \begin{array}{cc} 0 &      1 \\
          -1  &   q \end{array} \right]^k
            \left[ \begin{array}{cc} 1 &       -1\\
      1  &   1 \end{array} \right].
\end{equation}
 Chebyshev polynomials satisfy
a myriad of identities reflecting the group structure of
$\{A^n\}_{n \in \mathbb Z}$.   In particular $U_{2k+1}= V_{2k+1}W_{2k+1}$,
which justifies the way we index the polynomials of the third and fourth kind.

The number-theoretic properties of the values of Chebyshev polynomials
for integer $q$ were studied by many authors.
Lucas
%, \cite{L},
studied the more general case of
arbitrary integer determinant. Lucas theory is presented in the
books of  Ribenboim,\cite{R}, and Williams, \cite{W},
where extensive bibliography can be found.  
We propose a novel approach, and we give a complete presentation.
Our point of view is close to that of 
Schur,\cite{S}, Rankin,\cite{Ra}, and Arnold,\cite{A}.

We consider the Chebyshev  polynomials for  $n \geq 1$, and only 
for rational   values of the variable $q$.
For a prime number $p$ we will be using modular arithmetic  $mod \ p$
for rationals with denominators which are not divisible by
$p$. For a rational  $q= \frac{a}{b}$, with coprime
$a,b$, we have $q = 0 \ mod \ p$ if $p\mid a$ and $p\nmid b$.
We denote by $\mathcal D(q)$
the set of divisors of $b$, the denominator of $q$.

Let  $\Pi$ be the set of
odd primes.
For any rational  $q$ we introduce
the subsets of odd primes
$\Pi_s=\Pi_s(q)\subset \Pi, s = 0,1,2,\dots $. An odd prime
$p \in \Pi_0(q)$ if
there is an odd  $n$ such that $W_{n}(q)= 0 \ mod \ p$,
and it belongs to 
$\Pi_1(q)$ if
there is an odd $n$ such that $V_{n}(q)= 0 \ mod \ p$.
An odd prime
$p \in \Pi_{2+k}(q), k\geq 0,$ if
there is an odd $n$ such that
$C_{2^{k}n}(q)= 0 \ mod \ p$.
\begin{theorem}
  For every rational $q$ the sets
  $\Pi_s(q),\  s = 0,1,2,\dots$, are disjoint, and
  the set $\Pi\setminus \mathcal D(q)$ is partitioned by them,
  $\bigcup_{s =0}^{\infty}\Pi_{s}=\Pi\setminus\mathcal D(q)$.
\end{theorem}
We will denote this  partition of odd primes by $\rho(q)$.
There are no exclusions in the rational values of $q$,
however the partition is trivial, in having only one nonempty
set, if and only if $q=0,\pm 1,\pm 2$.

The Chebyshev polynomials are reducible over $\mathbb Z$. Their
irreducible factors are {\it Chebotomic polynomials}
$\Psi_k$  defined for $k\geq 3$ as follows, \cite{Y}, 
\[
\Psi_k(z+z^{-1}) = z^{-\frac{\varphi(k)}{2}}\Phi_k(z),
\]
where $\Phi_k(z)$ is the $k$-th cyclotomic polynomial,
and $\varphi(k)$ is the Euler function, equal to the degree
of $\Phi_k$.
The factorization of Chebyshev polynomials is given by
Yamagishi, \cite{Y}, for odd $n$ and arbitrary $l$, 
\[
W_n=\prod_{1< d|n}\Psi_d, \ \ \ V_n=\prod_{1< d|n}\Psi_{2d},
\ \ \ C_{2^ln} =\prod_{1\leq d|n}\Psi_{2^{l+2}d}.
\]
It transpires from these factorizations that the partition
$\rho(q)$ can be described as follows: an odd prime $p$
belongs to $\Pi_s(q)$ if and only if there is an odd $d$
such that $\Psi_{2^sd}(q) = 0 \ mod \ p$.

Substituting $q=z+z^{-1}$ into Chebyshev polynomials we obtain,
for arbitrary  $k$, 
\[
\begin{aligned}
W_{2k+1}(q) = \frac{z^{2k+1}-1}{z^{k}\left(z-1\right)}=&
  z^{k}+z^{k-1}+\dots +1+\dots +z^{-(k-1)}+z^{-k},\\
V_{2k+1}(q) = \frac{z^{2k+1}+1}{z^{k}\left(z+1\right)}=&
  z^{k}-z^{k-1}+\dots  (-1)^k \dots  -z^{-(k-1)} +z^{-k},\\
  U_{k}(q)= \frac{z^{k}-z^{-k}}{z-z^{-1}}=
  &z^{k-1}+  z^{k-3}+\dots +z^{-(k-3)}+z^{-(k-1)},\\   
  C_k (q) = &z^k+z^{-k}.
  \end{aligned}
\]

In general the partitions $\rho(q)$ are different for different $q$ with
notable exceptions. 
\begin{theorem}
  For $q_2 = q^2-2 =C_2(q)$ the partitions $\rho(q_2)$ and  $\rho(q)$
  are related as follows, $\Pi_{0}(q_2) = \Pi_{0}(q)\cup \Pi_{1}(q)$,
  $\Pi_k(q_2) = \Pi_{k+1}(q), k = 1,2,\dots $.

  If $q_n = C_{n}(q)$, for an odd $n$, then
  the partitions $\rho(q_n)$ and  $\rho(q)$ coincide,
  and $\Pi_j(q_n) = \Pi_{j}(q), j=0,1,2,\dots$.  

  For any $q$  the partitions $\rho(q)$ and $\rho(-q)$ coincide,
  with $\Pi_{0}(-q) = \Pi_{1}(q)$, and
  $\Pi_j(-q) = \Pi_{j}(q), j=2,3,\dots$.  
\end{theorem}
Theorem 2 has the following immediate consequence.
\begin{corollary}
  For any rational $q$ and natural $k\geq 1$, if  $q_{k} = C_{2^k}(q)$
  then 
  \[
  \Pi_0(q_{k}) = \bigcup_{j= 0}^k\Pi_{j}(q)=  \Pi_1(-q_{k})
   , \ \  \ \
  \Pi_1(q_{k}) = \Pi_{k+1}(q) = \Pi_0(-q_{k}).
  \]
\end{corollary}

For $q_1 = q^2-2 =C_2(q)$ we say that  $\rho(q_1)$ is the {\it square} of the
  partition $\rho(q)$, and that $\rho(q)$ is the {\it root} of the
  partition $\rho(q_1)$.   
  We say that $q$ is a {\it primitive trace}, and $\rho(q)$ is a
  {\it primitive} partition,
  if neither $\rho(q)$
  nor $\rho(-q)$ have a root.
  Clearly a rational trace $q$ is primitive if and only if
  neither $2+q$ nor $2-q$ is a rational square.

  Both $\rho(q)$
and $\rho(-q)$ have root partitions $\rho(w)$
  and  $\rho(z)$, respectively,
  if and only if $q = w^2-2=2-z^2$. In particular $w^2+z^2 =4$,
  which is rare. We will discuss this case in detail. 
Note that,  with this exception, the partition 
$\rho(-q_{k})$ from the Corollary does not have a root.

  By Theorem 2 it is enough to study the partitions $\rho(q)$ only
for primitive traces $q$.

For a subset $S$ of primes the  {\it prime density } of $S$ is denoted
by $|S|$ and defined as
\[
|S| = \lim_{n\to+\infty}
\frac{\#\{p\in S| p \leq n\}}{\#\{p\in \Pi| p \leq n\}}.
\]
Stevenhagen and Lenstra, \cite{S-L}, gave an elementary
introduction to the methods of
calculating prime  densities, with an outline of the history of the subject.
Let us note that prime density  is a finitely additive
set function, and it is not countably additive.

The calculation of prime densities of the sets in a partition $\rho(q)$
was accomplished in special cases by 
Hasse, \cite{H} and Lagarias, \cite{Lag}. The main
tool was the Chebotarev density theorem.
It was followed by the work of 
Ballot, \cite{B1},\cite{B2},\cite{B3},
Moree, \cite{M1},\cite{M2}, Moree and Stevenhagen, \cite{M-S}, 
Ballot and Elia, \cite{B-E}, where many other cases were 
studied, in all instances with the application of the
Chebotarev theorem.

We obtain the densities in full generality
thanks to the symmetry of the problem:
$-q\leftrightarrow q$, which was elusive in the classical
language of Lucas sequences. This hidden symmetry allows
the description of the  elements of the
partition inductively in a constructive way.
The calculation of densities is then achieved
by the application of the  Frobenius density theorem, \cite{S-L}.

We recall the Legendre symbol $(q|p)$, for a rational $q =\frac{a}{b}$,
and  an odd prime $p$ not dividing the denominator $b$:
$(q|p)=1$ if
$ab$ is a quadratic  residue $mod \ p$,  $(q|p)=-1$ if
$ab$ is a quadratic nonresidue $mod \ p$, and
$(q|p)=0$ if $q = 0 \ mod \ p$.

Let $\Pi_{*}=\Pi_{*}(q)=\bigcup_{k=2}^{+\infty}\Pi_{k}(q)$.
For any primitive  partition $\rho(q)$ the content of following table
is proven in Theorem 12 and Proposition 11.
\begin{center}
\begin{tabular}{|c|c|c|}
  \hline
  & $(2+q|p) = 1 $ &  $(2+q|p) = -1$ \\
  \hline
  $(2-q|p) = 1$ &
    \begin{tabular}{c|c}
    $(2|p) = 1 $ &  $(2|p) = -1 $ \\
  \hline
  ??? & $ \Pi_0 \cup \Pi_1$  
\end{tabular}
  & $\Pi_{1}$ \\
  \hline
   $(2-q|p) = -1$ & $\Pi_0$ &
  \begin{tabular}{c|c}
    $(2|p) = 1 $ &  $(2|p) = -1 $ \\
  \hline
  $\Pi_*\setminus \Pi_2$ & $\Pi_2$  
\end{tabular}\\
  \hline
\end{tabular}
\end{center}
In this table the sets corresponding to the  boxes
are contained in the respective elements of the partition $\rho(q)$.
We make no claim about the box with question marks.

We say that a pair of rational numbers $a,b$ is {\it generic} if
all the six numbers $a,b,c=ab,2a,2b,2c$ are not rational squares.
It is straightforward that a pair $a,b$ of rationals is generic
if and only if $[\mathbb Q\left(\sqrt{2},\sqrt{a},\sqrt{b}\right):\mathbb Q]
=8$.
A rational number $q$ is called {\it generic}
if the pair $2+q,2-q$ is generic. Note that by the definition
the partition $\rho(q)$ for a generic $q$ is by necessity primitive.

For a generic  $q$ it follows from this table by the Frobenius theorem
that $7/8$ of all primes are explicitly allocated
to the subsets $\Pi_0,\Pi_1,\Pi_{2}$ and $\Pi_{*}\setminus\Pi_{2}$.

We will refer to the four main boxes in this table as cells.
The upper-left cell is then refined inductively in an
infinite  sequence of tables
in Theorem 12. 
Let 
$\Omega_s^\pm =  
  \{p\in \Pi\ | \ C_{2^s}(a)= \pm q_0 \ mod \ p \
  \text{ for some}\   a \in \mathbb F_p \}$ for $s=1,2,\dots$.
%$\Omega_s^\pm =  
%  \{p\in \Pi\ | \exists a\in \mathbb F_p: \ C_{2^s}(a)= \pm q_0 \ mod \ p \
%   \}$ for $s=1,2,\dots$.
Here is the $k$-th table. 
\begin{center}
\begin{tabular}{|c|c|c|}
\hline
 & $p\in\Omega^+_k$ & $p\notin\Omega^{+}_k$ \\
 \hline
 $p\in\Omega^-_k$&  ??? & $\Pi_{1}$ \\
  \hline
  $p\notin\Omega^{-}_k$ & $\Pi_0$ & $\Pi_{*}$\\
\hline
\end{tabular}
\end{center}
Now every  prime is unambiguously assigned
to one of the sets $\Pi_0,\Pi_1,\Pi_*$.
These assignments combined deliver complete description
of the three subsets. We are able then to further
resolve the splitting of $\Pi_*$ in Theorem 13. This interpretation of
the partition $\rho(q)$ allows the application of
the Frobenius theorem, and the calculation of
the densities for all sets in the partition $\rho(q)$, for all values of $q$.

A sequence of disjoint sets of primes
$\{Y_j\}_{j\geq 0}$,
 is called {\it dyadic} if
the sets $S_k=\bigcup_{j=k}^{\infty}Y_j$, have prime densities,
and $|S_{k+1}| = |S_k|/2, k=0,1,2,\dots$. It follows from this definition
that, by finite
additivity, the sets $Y_j$ have prime densities, and
also $|S_0|= \sum_{j=0}^{\infty}|Y_j|$.
\begin{theorem}
For any generic rational $q$ 
\[
  |\Pi_0|= |\Pi_1| =\frac{1}{3},
  \]
  and the sequence $\{\Pi_s\}_{s\geq 2}$ is dyadic.
  \end{theorem}
The results of Ballot,\cite{B2}, and Ballot and Elia,\cite{B-E},
can be translated as   $|\Pi_0|=|\Pi_1|=\frac{1}{3}$,
for most rational $q$. Similar conclusions can be found in
the unpublished thesis of Ljujic,\cite{Lj}.
Our genericity requirements are less stringent.
We  put all possible violations of genericity into three cases,
and calculate then the densities in the following
theorem. 
\begin{theorem}
If (A): \  $2(2+q)$ is  a rational square,
  and  neither $4-q^2$ nor $2-q$
is a rational square,
then
  \[
  |\Pi_0|= |\Pi_1| =\frac{7}{24},\ \  |\Pi_2| =\frac{1}{3},
  \]
  and the sequence $\{\Pi_s\}_{s\geq 3}$ is dyadic.
  
If (B): \ $2(4-q^2)$ is  a rational square, and neither $2+ q$ nor $2-q$
is a rational square, 
then
  \[
  |\Pi_0|= |\Pi_1| =\frac{7}{24},\ \  |\Pi_2| =\frac{1}{12}, 
  \]
  and the sequence $\{\Pi_s\}_{s\geq 3}$ is dyadic.

  If (C): \ $4-q^2$ is a rational square, and neither $2+ q$ nor $2(2+q)$
is a rational square,  then
\[
|\Pi_0|= |\Pi_1| =\frac{1}{6},
  \]
  and the sequence $\{\Pi_s\}_{s\geq 2}$ is dyadic.  
\end{theorem}
If a rational $q$ satisfies the conditions in the case (C)
it is called {\it circular primitive}. There are primitive
but not circular primitive values of $q$, which are left out
of Theorem 5. They are described explicitly in Section 4,
in a way similar to Theorem 2. As a result we provide a way
to obtain densities of the sets in the partition
$\rho(q)$ for any rational  $q$.

Let us outline the contents of the paper.
In Section 2 we prove 
Theorems 1 and 2. The highlight of the proofs is that
they follow directly from appropriate identities
satisfied by Chebyshev polynomials.
 In Appendix B
 we give the proofs of all these identities
 using the algebra of $SL(2,\mathbb C)$.
 Although they can be found in other places,
 the proofs provided illuminate our approach to the subject.

In Section 3 we prove the following Euler's Criterion
in $SL(2,\mathbb Q)$: {\it 
for any matrix $A\in SL(2,\mathbb Q)$, with the trace $q$,
and   for any odd prime $p\notin \mathcal D(q)$,
  and $\delta= q^2-4$, if $(\delta|p) \neq 0$ then 
$
  A^{\frac{p-(\delta|p)}{2}} = (q+2|p)\ I \ \ mod \ p$}.
This fact, proven in Theorem 7, serves as the bridge between
the algebra of Chebyshev polynomials and the number-theoretic
properties of the partitions $\rho(q)$.
%\begin{theorem}
%  For any matrix $A\in SL(2,\mathbb Q)$, with the trace $q$,
%and   for any odd prime $p\notin \mathcal D(q)$,
%  and $\delta= q^2-4$, if $(\delta|p) \neq 0$ then 
%\[
%A^{\frac{p-(\delta|p)}{2}} = (q+2|p)\ I \ \ mod \ p.
%\]
%If $(\delta|p) = 0$ then $A^{p} = \frac{q}{2}I = \pm I \  mod \ p$.
%\end{theorem}
Using this tool we obtain in Theorems 12 and 13 
the aforementioned inductive splitting
of primes into four cells.

With the inductive  description of the partition $\rho(q)$ in hand
we calculate in  Section 4 the densities of Theorems 4 and 5,
conditional on the densities of certain subsets.
These subsets (some cells, and  their special subsets)
turn out to be the sets of primes for which appropriate polynomials
split completely over $\mathbb F_p$. The establishment of these
facts is the subject of the next Section 5.
By the Frobenius theorem these sets have prime densities
equal to the inverse of the order of the Galois group of the
polynomials. In Appendix A we calculate these orders using
rather elementary Galois Theory involving towers of quadratic
extensions. As a result all of the orders  are powers of $2$.

In Section 6 we translate our results into the 
language of Lucas sequences. It transpires that our
approach simplifies  some aspects of the classical theory.
It seems that the phenomenon of twin partitions
was overlooked in the past, although it is
hard to be sure due to the enormous volume of the
literature of the subject.

Finally in Section 7 we give an application of our results to
the number-theoretic properties of simple dynamical
systems: rotations  and certain interval maps.
This is in the spirit of studying dynamics where
$\mathbb R$ is replaced with $\mathbb Q$, or
finite fields $\mathbb F_p$.

We thank  Weronika Flak for her
infinite patience with reading the earlier versions
of the paper. The progress was made thanks to her
efforts.

We benefited greatly from friendly discussions with Roman Marsza\l ek
and Maciej Ulas.

\section{The algebra of Chebyshev polynomials and the proofs of Theorems
  1 and 2}

The Chebyshev polynomials are given by the
following expansions, which can be proven by induction.
\[
C_n(q) = 
\sum_{s=0}^{[\frac{n}{2}]}(-1)^s\frac{n}{n-s} {{n-s} \choose s} q^{n-2s},\ 
U_{n+1}(q)=\sum_{s=0}^{[\frac{n}{2}]}(-1)^s {{n-s} \choose s} q^{n-2s}.
\]
The Chebyshev polynomial  $C_{n}(q)$
contains only even powers of $q$  for even $n$,
and only odd powers for odd $n$. Hence $q$ is a factor of the polynomial
$C_{n}(q)$ for odd $n$.
 The degree of  $U_{n}(q)$ is $n-1$, and it contains only odd powers
 for even $n$, and only even powers for odd $n$.

 It follows that $C_n(-q) = (-1)^nC_n(q), \ U_{n+1}(-q) = (-1)^{n}U_{n+1}(q)$
and $V_{2n+1}(-q) = (-1)^{n}W_{2n+1}(q)$. This gives us obviously
the following equalities, which are part of Theorem 2.
For any $q\in \mathbb Q$
  \begin{equation}
    \Pi_{0}(-q)= \Pi_{1}(q), \ \ \Pi_{1}(-q) =\Pi_{0}(q),\ \
    \Pi_s(-q) = \Pi_{s}(q),\ s = 2,3 \dots.
  \end{equation}
  Note that neither the formulation of (4) nor the proof
  requires the knowledge
that the sets form a partition. 
In the future we will refer to the values $q$ and $-q$ as the {\it twins}.

 Substituting $A=I$ in (1) we get that
 $C_n(2) =2, \ U_n(2) = n$.
 Further $ V_{2n+1}(2) = 1, W_{2n+1}(2) = 2n+1$.

 We now formulate nine identities for Chebyshev polynomials
 that we will be using throughout the paper. Their proofs are given
 in Appendix B.
 For odd $n,m$ and arbitrary natural $k,l$ we have the following
 formulas.
\[
   \begin{aligned}
     (F1)   \ \ \  \ \ \ \ \ \ \
     &C_{kl}= C_k\left(C_l\right)=C_l\left(C_k\right),\\
     (F2)  \ \ \  \ \ \ \ \ \ \
     &U_{kl}=U_{k}U_{l}\left(C_k\right),\\
     (F3)   \ \ \  \ \ \ \ \ \ \     &U_{k+1}C_{k}-C_{k+1}U_{k}=2,
     \ \ W_nV_{m-2}-V_nW_{m-2}=2\\
     (F4)   \ \ \  \ \ \ \ \ \ \     &U_{n} =V_{n}W_{n},
     \ \ U_{2k} = C_kU_k,\\
     (F5)   \ \ \  \ \ \ \ \ \ \     &C_{n}-2 = (q-2)W_{n}^2, \ \
     C_{n}+2 = (q+2)V_{n}^2,\\
     (F6)   \ \ \  \ \ \ \ \ \ \     &W_{nm} =W_{n}W_{m}(C_{n}), \ \
     V_{nm} =V_{m}(C_{n})V_{n},\\
 (F7)   \ \ \  \ \ \ \ \ \ \       &C_{n}(q)=qV_{n}(q^2-2),
\ \ \ U_{n}(q) = W_{n}(q^2-2),\\
(F8)     \ \ \  \ \ \ \ \ \ \     &C_{n}(q)=
qU_{n}(\sqrt{4-q^2}), \ \ C_{n}(\sqrt{2+q})=
\sqrt{2+q}U_{n}(\sqrt{2-q}),\\
(F9)     \ \ \  \ \ \ \ \ \ \     &C_{k}^2(q)+(4-q^2)U_{k}^2(q) =4. 
   \end{aligned}
\]

We proceed with the proofs of Theorems 1 and 2.

Let us first argue that for $q=\frac{a}{b}$,
any odd prime $p$ which does not
divide $b$, belongs to one of the
sets $\Pi_s(q), s \geq 0$. 
Matrices in $SL(2,\mathbb Q)$  which contain only
rationals with denominators equal to  powers of $b$, form a subgroup of 
$SL(2,\mathbb Q)$. This subgroup can be  factored 
homomorphically onto $SL(2,\mathbb F_p)$. It follows that for 
$A = \left[ \begin{array}{cc} 0 & 1 \\ -1  & q \end{array} \right]$ the matrices  $\{A^n\}_{n\in \mathbb Z} \ mod \ p $, form
a subgroup of $SL(2,\mathbb F_p)$, hence a finite group.
In particular there is a natural number $k$ such that $A^k = I \ mod \ p$,
hence by (2) we have $U_k(q)= 0 \ mod \ p$.
This gives us the claim,
since by (F4) we have for any $k = 2^sn$, with odd $n$, 
$ U_{k} = W_nV_nC_{n}C_{2n}C_{4n}\dots C_{2^{s-1}n}$.

Now we will address the disjointness of the sets    
$\Pi_s(q), s=0,1,\dots$. 
We first show that $\Pi_0$ and $\Pi_1$ are disjoint.
By (F6) for any  odd $n$ and  $m$,
$W_{nm}(q)=W_{n}(q)W_m\left(C_n(q)\right)$ and
$V_{nm}(q)=V_{m}(q)V_n\left(C_m(q)\right)$.
It follows that if $W_n(q)= 0 \ mod \ p$ and
$ V_m(q)=0 \ mod \ p$ 
then   $W_{nm}(q)= 0 \ mod \ p$ and
$ V_{nm}(q) = 0 \ mod \ p$.
By (F3) we would get
\[
2= W_{nm}(q)V_{nm-2}(q)-V_{nm}(q)W_{nm-2}(q) = 0 \ mod \ p,
\]
which is a contradiction.

To proceed we will first establish the first part of Theorem 2,
namely that for $q_2=C_2(q)$ we have $\Pi_{0}(q_2)
= \Pi_{0}(q)\cup\Pi_{1}(q), \ \Pi_{s}(q_2) = \Pi_{s+1}(q), s=2,3,\dots $.
Note that the formulation of these set equalities does not
require the validity of Theorem~1.

By (F1) for any  $k$ we have $C_{2k}(q) =C_{k}(q_2)$ which shows
that

$\Pi_{s+1}(q)=\Pi_{s}(q_2)$, for any $s\geq 2$. 

By (F7) we have  $qV_n(q_2) = C_n(q)$ for any odd $n$.
It shows  that $\Pi_{1}(q_2)=\Pi_{2}(q)$ except for the status of
the divisors of the numerator of $q$, which are always
in $\Pi_{2}(q)$. We still need to 
establish that they belong to
$ \Pi_{1}(q_2)$. By the rules of modular arithmetic
if $q =0 \ mod \ p$ then $q^2-2 =-2 \ mod \ p$ and
$V_{k}(q_2) = V_{k}(-2) \ mod \ p$.
Now we conclude with $V_{p}(-2) =\pm W_p(2)=\pm p   = 0 \ mod  \ p$.

Combining (F4) and (F7) we get $ W_n(q_2)=W_n(q)V_n(q)$
for any odd $n$. This proves that $\Pi_0(q_2) = \Pi_0(q)\cup  \Pi_1(q)$.

Now we can finish the proof that the sets $\Pi_k(q)$ are disjoint.
We know already that for any $q$ the sets $\Pi_0(q)$ and $\Pi_1(q)$
are disjoint. We prove by induction on $k$ that
the sets $\Pi_0(q), \Pi_1(q), \dots, \Pi_k(q)$ are disjoint.
By Corollary 3
\[
\Pi_0\left(C_{2^k}(q)\right)
=\Pi_0(q)\cup \Pi_1(q)\cup \dots\cup \Pi_k(q), \ \
\Pi_1\left(C_{2^k}(q)\right) = \Pi_{k+1}(q). 
\]
Since $\Pi_0$ and $\Pi_1$ are always disjoint we get the inductive step.

Theorem 1 is proven.

To finish the proof
of Theorem 2 we will use Theorem 1.
By (F1) we have  that $C_k\left(C_n(q)\right) =
C_{nk}(q)$. It follows that 
$\Pi_*\left(C_n(q)\right) \subset\Pi_*(q)$
for any $n$,
however for odd $n$ we can also claim that 
$\Pi_s\left(C_n(q)\right) \subset \Pi_s(q), \ s =2,3,\dots$.

By (F6) we get for odd $n$ and $m$ that $V_n(q)V_m\left(C_n(q)\right) =
V_{nm}(q)$. It follows  that if   $V_m\left(C_n(q)\right) =0 \ mod \ p $
then
$V_{nm}(q) =0 \mod p $. Hence   
$\Pi_1\left(C_n(q)\right) \subset \Pi_1(q)$.  In a similar way we obtain
that $\Pi_0\left(C_n(q)\right) \subset \Pi_0(q)$.

All these inclusions between elements of
the partitions $\rho\left(C_n(q)\right)$
and $ \rho(q)$ are possible only if the partitions coincide.
Theorem 2 is proven.

\section{Euler's Criterion in $SL(2,\mathbb Q)$}

\begin{proposition}
    For any $q\in \mathbb Q$ and any odd $p\notin \mathcal D(q)$,  
    we have
    \begin{equation}
      \begin{aligned}
        &C_p(q) = q  \ mod \ p,\\
        &V_p(q) = (q+2|p)\ mod \ p, \ W_p(q) = (q-2|p) \ mod \ p,\\
        &U_p = (q^2-4|p)\ mod \ p.
      \end{aligned}
        \end{equation}
\end{proposition}
\begin{proof}
  The first line follows from the expansion of $C_k$ above, when we note that
  for $k=p$ all the coefficients $\frac{p}{p-s}{p-s \choose s}$
  are divisible by $p$ with the exception of $s=0$, the leading coefficient
  equal to $1$. Hence $C_p(q) = q^p = q \ mod \ p$ by the Fermat's little
  theorem.

  Further by (F7) 
  we have $V_p(q) = (\sqrt{q+2})^{-1}C_p(\sqrt{q+2})$.
  Note that $t^{-1}C_p(t)$ is a polynomial in $t^2$ with 
  all
  coefficients divisible by $p$, except for the leading coefficient
  equal to $1$.
  It follows that $V_p(q) = (\sqrt{q+2})^{p-1} \ mod \ p$. By the Euler's
  criterion $V_p(q) = (q+2|p)\ mod \ p$. Now
  \[
  W_p(q) = (-1)^{\frac{p-1}{2}}V_p(-q) = (-1)^{\frac{p-1}{2}}(-q+2|p)
  = (q-2|p) \ mod \ p.
  \]

  The last line is a direct consequence of the second line and (F4).
  \end{proof}

The following is the Euler's criterion in $SL(2,\mathbb Q)$.
\begin{theorem}
  For any matrix $A\in SL(2,\mathbb Q)$, with the trace $q$,
and   for any odd prime $p\notin \mathcal D(q)$,
  and $\delta= q^2-4$, if $(\delta|p) \neq 0$ then 
\[
A^{\frac{p-(\delta|p)}{2}} = (q+2|p)\ I \ \ mod \ p.
\]
If $(\delta|p) = 0$ then $A^{p} = \frac{q}{2}I = \pm I \  mod \ p$.
\end{theorem}
\begin{proof}
  We will first prove that $A^{p-(\delta|p)} = I$. By  (5)
  $C_p(q) = q \ mod \ p$.
    Since
    \[
    C_p = U_{p+1}-U_{p-1} = q \ mod \ p, \ \ \ \ \ 
          U_{p+1}+U_{p-1} =    qU_{p},
    \]
    we get
    \begin{equation}
      2U_{p+1} = q\left(U_{p}+1\right) \ mod \ p,\ \  
      2U_{p-1} = q\left(U_{p}-1\right) \ mod \ p.
    \end{equation}
    Now we use the last line of (5).
    If $(\delta|p) =1$ then $U_{p} = 1 \ mod \ p$ and we get by (6) that
    $U_{p-1} = 0 \ mod \ p$. Consequently $A^{p} = U_pA-U_{p-1}I=A \ mod \ p$.

    Similarly, if $(\delta|p) =-1$ then $U_{p} = -1 \ mod \ p$
    and by (6)
    $U_{p+1} = 0 \ mod \ p$. Consequently $A^{p+1} = U_{p+1}A-U_pI = I\ mod \ p$.

    Finally if $(\delta|p) =0 $ then  $U_p = 0 \ mod \ p$ and 
    $2U_{p-1} = -q \ mod \ p$,
    which leads to $A^{p} = U_pA-U_{p-1}I=\frac{q}{2}I \ mod \ p$.    

    To finish the proof let $p-(\delta|p) = 2r$. In the following
    all the relevant equalities are $mod \ p$. Since
    $I =A^{2r} = U_{2r}A-U_{2r-1}I $ then $U_{2r} =0$. By (F4) $U_{2r} = U_rC_r$
    so that either $U_r=0$ or $C_r=0$. The last equality is impossible because
    $2=C_{2r}=C_r^2-2$. It follows that $A^r=-U_{r-1}I=sI$ where $s=\pm 1$,
    and hence $U_{r+1}=s, U_r=0, U_{r-1}=-s$. We finally observe that
    $V_{2r+1}=U_{r+1}- U_r=s$ and     $V_{2r-1}=U_{r}- U_{r-1}=s$. Since
    $p=2r+(\delta|p)$ then $(q+2|p)=V_p=s$.
\end{proof}
Following the Lucas theory let us define for a rational $q$
the {\it index of appearance $\xi(p)$},
for any  odd prime $p\notin \mathcal D(q)$, as the smallest natural number
$k$ such that $U_k(q)=0 \ mod \ p$. It is clear
that $\xi(p)$ is well defined, it follows for example from Theorem 7.
\begin{proposition}
  For any rational $q$ and any $p\notin \mathcal D(q)$ the index of
  appearance $\xi(p)$ is odd if and only if
  $p\in \Pi_0(q)\cup\Pi_1(q)$.
  
If $p\in \Pi_0(q)$ then  $C_\xi = 2 \ mod \ p$
and $U_{\xi+1} = 1 \ mod \ p$.

If $p\in \Pi_1(q)$ then
$C_\xi = -2 \ mod \ p$
and  $U_{\xi+1} = - 1 \ mod \ p$.

If $\xi(p) =2k$ then $C_k=0 \ mod \ p$, $C_\xi = - 2\ mod \ p$
and $U_{\xi+1}= -1 \ mod \ p$.
\end{proposition}
\begin{corollary}
  For any matrix $A\in SL(2,\mathbb Q)$, with the trace $q$,
  if  $p\in \Pi_0(q)$ then  $A^{\xi(p)} = I  \ mod \ p$, and
  if  $p\notin \Pi_0(q)$ then  $A^{\xi(p)} = -I  \ mod \ p$.
\end{corollary}
\begin{corollary}
  If $(\delta|p) \neq 0$ then $2\xi(p)$ divides  $p -(\delta|p)$.
  If $(q+2|p)=-1$ then
  $\frac{p-(\delta|p)}{2\xi(p)}$ is odd. 
  For $p \notin \Pi_0$,  if $\frac{p-(\delta|p)}{2\xi(p)}$ is odd
  then $(q+2|p)=-1$. 
  
If $p -(\delta|p)  = 2 \ mod  \ 4$  then $\xi(p)$ is odd.
If  $p -(\delta|p)=2r$ for a prime $r$
then $\xi(p)=r$. 
If  $(\delta|p)= 0$ then $\xi(p)=p$.
\end{corollary}

\begin{proof}

  If   $\xi(p)$ is odd then $0= U_\xi = W_\xi V_\xi \ mod \ p$.
  It follows that 
  $p\in \Pi_0(q)\cup\Pi_1(q)$.

  If $\xi(p) =2k$ then 
  $0= U_{2k} = U_k C_k \ mod \ p$ and hence $C_k=0 \ mod \ p$,
  by the minimality of $\xi$.

Let us note that for the matrix $A$ in Corollary 9,

$A^{\xi(p)} =U_{\xi(p)}A- U_{\xi(p)-1}I = U_{\xi(p)+1}I \ mod \ p$.
Hence $U_{\xi+1}^2=1$ and Corollary 9 will follow from the
Proposition.

If $p\in \Pi_0(q)$ then by (F5) $C_\xi = 2 \ mod \ p$,
and hence $U_{\xi+1} = 1 \ mod \ p$. Similarly if $p\in \Pi_1(q)$ then
by (F5) $C_\xi = -2 \ mod \ p$
and  $U_{\xi+1} = - 1 \ mod \ p$.

If $\xi(p)=2k$        then we have  $A^{2k} = U_{\xi+1}I \ mod \ p$ and the
trace of this matrix is equal to $C_{2k}= 2U_{\xi+1}$. On the other hand 
since $U_{2k}=U_kC_k=0 \ mod \ p$ we get  $C_k=0 \ mod \ p$, and hence
$C_{2k}=C_k^2-2 =-2  \ mod \ p$. It follows that $A^{2k}= -I$.
Corollary 9 is proven.

To prove Corollary 10 we observe that $\xi(p)$ is the first $k$
such that  $A^k=
\left[ \begin{array}{cc} 0 & 1 \\ -1  & q \end{array} \right]^k =
\left[ \begin{array}{cc} -U_{k-1} & U_k \\ -U_k  & U_{k+1} \end{array} \right]
=U_{k+1}I\ mod \ p$. At the same time by Theorem 7 if $(\delta|p) \neq 0$ then
$A^{\frac{p-(\delta|p)}{2}}=\pm I$.
It follows that  $2\xi(p)$ must divide  $p-(\delta|p)$.

Further let $ p-(\delta|p)=2k\xi(p)$. If $(q+2|p)=-1$ then by Theorem 7
$A^{k\xi} =-I$,
and hence also $A^{\xi} =-I$ and $k$ is  odd.
Conversely  if $p\notin \Pi_0$ then $A^{\xi(p)}=-I$ by Corollary 9, and so
if $k$ is odd then $A^{k\xi} =-I$, and $(q+2|p)=-1$ by Theorem 7.

Finally,  if $p-(\delta|p) = 2 \ mod \ 4$ then  $\xi$ must be odd,
and if $p -(\delta|p)=2r$ for a prime $r$
then $\xi(p)=r$.

If $(\delta|p) = 0$ then by necessity $\xi(p) = p$.
\end{proof}

We proceed with the formulation of our main result.
We start with a simple general observation valid
for all values of $q$. 
\begin{proposition}
For any rational $q$ we have
  \[
  \Pi_*\setminus \Pi_2\subset \{p\ | \ (2|p) =1 \},
  \ \ \ \Pi_2\subset \{p\ | \ (2|p) =(2+q|p)=(2-q|p) \},
  \]
 $\{p\in \Pi\ | \ (2+q|p) =0 \} \subset
    \Pi_1$ \ and \ $\{p\in \Pi\ | \ (2-q|p) =0 \} \subset
    \Pi_0$.      
\end{proposition}  
\begin{proof}
  To prove the first claim we use the formula
  $C_{2k}(q) = C_k^2(q) -2$. To prove the second claim  let  $p\in \Pi_2$,
  then there is an odd $n$ such that  
  $C_n(q)= 0 \ mod \ p$. It follows from (F5) that
  $(2|p) =(2-q|p)$ and $(2|p) =(2+q|p)$.

  The last claims follow from
  $W_p(2) = p$ and $V_p(-2) = (-1)^{\frac{p-1}{2}}W_p(2)$. 
\end{proof}

We are going to fix a particular value of the trace $q_0$,
and we remove the divisors of $\delta= -(2+q_0)(2-q_0)$ from further
considerations. By Proposition 11 we know where they belong,
and without them $\widehat p =\widehat p(q_0,p) =
\frac{p-(\delta|p)}{2}$ is well defined.
To ease the notation,  from now on $\Pi$ is understood as the set of
odd primes without the divisors of the numerator, or the denominator,
of $\delta$.

With that convention we introduce the following subsets in $\Pi$
\[
\Omega_s^\pm =  
  \{p\in \Pi\ | \ C_{2^s}(a)= \pm q_0 \ mod \ p \
  \text{ for some}\   a \in \mathbb F_p \}, \ \ s=1,2,\dots .
\]
  The sequences $\Omega_s^\pm, s=0,1,2,\dots,$ are nested:
  $ \Pi=\Omega_0^\pm \supset \Omega_{1}^\pm  \supset \Omega_{2}^\pm  \supset
  \dots$.

  Indeed if $C_{2^{s+1}}(a)= \pm q_0 \ mod \ p$,  for some $a \in   \mathbb F_p$,
    then for $b= C_2(a)=a^2-2$
    we have $C_{2^{s+1}}(a) = C_{2^{s}}(b)$, which proves that
    $\Omega_{s+1}^\pm\subset \Omega_{s}^\pm$.

  Let further
  \[
R_k =\Omega^+_{k}\cap \Omega^{-}_k,\ \ 
Z_k =R_{k-1}
\setminus  \left(\Omega^+_{k}\cup \Omega^{-}_k\right).
\]
\begin{theorem}
  If $2+q$ and $2-q$ are not rational squares then for $k=1,2,\dots$, 
  \[
  \begin{aligned}
    &\left(\Omega_{k}^+ \setminus \Omega_{k}^-\right)\cap R_{k-1} 
    \subset \Pi_0\cap \{p\ | \ 2^{k-1}\ | | \ \widehat p\},\\
    &\left(\Omega_{k}^- \setminus \Omega_{k}^+\right)\cap R_{k-1} 
    \subset \Pi_1\cap    \{p\ | \ 2^{k-1}\ | | \ \widehat p\},\\
    &Z_{k} \subset
    \Pi_*\cap  \{p\ | \ 2^{k}\ | \ \widehat p\},
    \ \ R_k \subset  \{p\ | \ 2^{k}\ | \ \widehat p\}.
  \end{aligned}
  \]
\end{theorem}
By this Theorem  every element in the subset
$R_{k-1}\setminus R_k, k=1,2,\dots$, 
is unambiguously assigned  to one of the sets $\Pi_0,\Pi_1,\Pi_*$.
These assignments combined deliver complete description
of the three subsets. 
This can be illustrated by 
the following table describing the assignment in the set $R_{k-1}$.

\

\begin{center}
\begin{tabular}{|c|c|c|}
\hline
 & $p\in\Omega^+_k$ & $p\notin\Omega^{+}_k$ \\
 \hline
 $p\in\Omega^-_k$  & $2^{k}\ | \
 \widehat p$ & $\Pi_{1},\ 2^{k-1} | | \widehat p$ \\
  \hline
  $p\notin\Omega^{-}_k$ & $\Pi_0,\ 2^{k-1} | | \widehat p$ &
  $\Pi_{*}, \ 2^{k}\ | \
  \widehat p$\\
\hline
\end{tabular}
\end{center}

Theorem 12 can be further refined with more detailed assignments
in $Z_k$. 
\begin{theorem}
  If $2+q$ and $2-q$ are not rational squares then for $k=1,2,\dots$, 
  \[
  Z_k\cap  \{\ p\  | \ 2^{k+s}\ | | \ \widehat p\ \}
  \subset \Pi_{s+2},
     \]
     \end{theorem}
Now the complete partition $\rho(q_0)$ is covered. We will give a joint
proof of Theorems 12 and 13.
\begin{proof}
  The proof is inductive. 
  We start with $k=1$. We have $\Omega_1^-\setminus \Omega_1^+
  =  \{p\ |  \ (2+q_0|p) =-1, (2-q_0|p) =1\}$. It follows by Theorem 7
  that for $p$ in this set $C_{\widehat p}(q_0) =-2$ and
  $C_{\widehat p}(-q_0) =2$. Obviously then $\widehat p$ must be odd,
  and $p\in \Pi_1$. 

  We have established all the required information about the
  upper-right cell in the table. By the symmetry $-q\leftrightarrow q$ we
  obtain immediately the respective information about the
  lower-left cell.

  For $p\in 
 \Omega^+_1\triangle\Omega^-_1=\Pi\setminus \left(R_1\cup Z_1\right)=
  \{ p \in \Pi \ | \ (2+q_0|p)(2-q_0|p) =-1 \}$
  we have   
  $ (\delta|p)=(-1|p)(2+q_0|p)(2-q_0|p) = -(-1|p)$, and further
  $\widehat p = \frac{p-(\delta|p)}{2}=\frac{p+(-1|p)}{2}=1     \ mod \ 2$.
  This set will not appear in the rest of the proof. In the
  complementary subset $\{ p \in \Pi \ | \ (2+q|p)(2-q|p) =1 \}$
  we have   $ (\delta|p)=(-1|p)$ and
  $\widehat p = \frac{p-(-1|p)}{2}=0 \ mod \ 2$. One consequence is
  the content of the upper-left cell in the table.

For  $p \in Z_1
=  \{p\ |  \ (2+q_0|p) =-1= (2-q_0|p)\}$ we argue similarly
  that $C_{\widehat p}(q_0) =-2$ and
  $C_{\widehat p}(-q_0) =-2$. It follows that  $\widehat p$ must be even,
  and $C_{\frac{\widehat p}{2}}(q_0) =0$, so that  $p\in \Pi_*$.

  To obtain the  claim from Theorem 13, let $ r =\frac{\widehat p}{2}$
  and $2^s| | r$. We have 
  $C_r(q_0) = 0$ and hence $p\in \Pi_{s+2}$. 

  We proceed with the inductive step, assuming that the table is valid
  for $k-1$ we examine the subset $R_k$.
  For $p\in R_{k}$ we have $x,y\in\mathbb F_p$ such that
  $C_{2^{k}}(x) =q_0$ and
  $C_{2^{k}}(y) =-q_0$. We claim that
  \[
  (\delta(x)|p)=(\delta(y)|p)=
  (\delta(q_0)|p) = (-1|p).
  \]
  Indeed $\delta(C_2(x)) = x^2\delta(x)$ for any $x$,
  and hence $(\delta(C_2(x))|p)=(\delta(x)|p)$. Induction finishes the proof.

  The claim has an important consequence that $\widehat p$ is the same for
  $x,y$ and $q_0$.

  For $p\in  R_{k} \setminus \Omega_{k+1}^+$   there is no
  $a\in \mathbb F_p$ such that  $C_2(a)=x$. Indeed we would have then
  $q_0=C_{2^{k}}(x) =C_{2^{k+1}}(a)$ which contradicts the assumption that
  $p$ is not in  $\Omega_{k+1}^+$. Hence $(2+x|p)=-1$.

  In the same spirit for $p\in  R_{k} \cap \Omega_{k+1}^-$ we can choose
  the element $a\in \mathbb F_p$ so that  $C_{2^{k+1}}(a)=-q_0$,
  and consequently for $y=a^2-2$ we have $C_{2^{k}}(y)=C_{2^{k+1}}(a)=-q_0$
  and $(2+y|p)=1$.

  For $p\in R_{k} \cap \left(\Omega_{k+1}^-\setminus \Omega_{k+1}^+\right)$
  we have thus the elements
  $x,y\in \mathbb F_p$ such that $(2+x|p)=-1, (2+y|p)=1$, and
  $C_{2^{k}}(x) =q_0, C_{2^{k}}(y) =-q_0$. It follows then
  by Theorem 7 that $C_{\widehat p}(x) =-2$ and $C_{\widehat p}(y) = 2$.
  By the inductive assumption
  $\widehat p = 2^{k}r$ for some natural $r$ and
  we get 
\[
-2= C_{\widehat p}(x) =C_{r}(C_{2^{k}}(x)) =C_{r}(q_0),
C_{r}(-q_0)=C_{r}(C_{2^{k}}(y)) =C_{\widehat p}(y) =2.
\]
It follows that $r$ is odd, $p\in \Pi_1$ and $2^{k}\ | | \ \widehat p$.

The claims for $\Omega_{k+1}^+\setminus \Omega_{k+1}^-$ will follow
by exchanging the roles of $q_0$ and $-q_0$.

By the same argument for $p\in  Z_{k+1}$
we have the elements $x,y \in\mathbb F_p$ such
that $C_{\widehat p}(x) =-2= C_{\widehat p}(y)$, and further
as before $C_{r}(q_0)=-2= C_{r}(-q_0)$ for $r=2^{-k}\widehat p$.
It follows that $r$ is even and $C_{\frac{r}{2}}(q_0)=0$.
For $s$ such that $2^s| | \frac{r}{2}$ we have
$p\in \Pi_{s+2}$.

Finally for $p\in R_{k+1}$
we have $C_{\widehat p}(x) =2 =C_{\widehat p}(y)$ and $C_{r}(q_0)=2= C_{r}(-q_0)$.
So that $r$ is again even. It follows that 
$2^{k+1} | \widehat p$.
The induction step is complete.
  \end{proof}

Let $\Gamma_s= \{p\ |\  p = \pm 1 \ \mod 2^{s+2}\}$.
 Clearly
 all odd primes belong to $\Gamma_0$ and $\Gamma_1= \{p\ | (2|p) = 1 \}$.
 Further $\Gamma_s= \{p\ |\ 2^{s+2}| p-(-1|p) \}, s\geq 0 $. 
Theorems 12 and 13 can be applied to $q_0=0$
and, using the irreducibility of the polynomial $C_{2^s}(x)$
over $\mathbb Q$, we get
\begin{proposition}
  \[
  \begin{aligned}
    \Gamma_s&= \{p\ | \ C_{2^s}(a)=0  \ mod \ p \
    \text{ for some}\   a \in \mathbb F_p \}\\
    &=\{p\ | \ C_{2^s}(z)\ 
    \text{splits into linear factors over }\  \mathbb F_p \},\ s\geq 1.
\end{aligned}
    \]
  \end{proposition}

When $\delta$ is a rational square then the splitting of 
$R_0$ into the four cells from Theorem 12, is somewhat simpler since then
$\widehat p = \frac{p-(-\delta|p)}{2} =\frac{p-(-1|p)}{2} = 0\ mod \ 2$
for all odd $p$. It leads to $R_1\cup Z_1 = \{p\ |\ p=  1 \ mod \ 4 \}$ and 
$\Omega_0^+\triangle \Omega_0^- = \{p\ |\ p=  3 \ mod \ 4 \}$,
which means that this part of the splitting does not depend on $q_0$.

\section{Prime densities}
The calculation of prime densities from Theorems 4 and 5
will be achieved  in several stages.
In the next Section we establish that the sets
$R_k,\Omega_{k+1}^\pm\cap R_k$, 
and  $R_k\cap \Gamma_{k+j}, k,j\geq 0$ 
are  sets of primes $p$
such that certain polynomials split completely over $\mathbb F_p$.
By the Frobenius theorem such a set has prime density
equal to the inverse of the order of the Galois group of the
polynomial. In Appendix A we calculate these orders using
towers of quadratic extensions.
As a result all of the orders  are powers of $2$. 
In this way  we will have established the following Proposition.
\begin{proposition}
    If $q_0$ is generic then 
  \begin{equation}
   |\Omega_{k}^\pm\cap R_{k-1}| =2^{-(2k-1)} , |R_{k}|=2^{-2k} 
    \ \ k=1,2,\dots  .
  \end{equation}
  
  In the non-generic cases (A) and (B)
  we have $\Omega_{2}^{+}\cap R_{1}= \Omega_{2}^{-}\cap R_{1} =R_2$,
  and (7) is replaced by 
  \[
  \begin{aligned}
&|\Omega_{1}^\pm| =2^{-1} , |R_{1}|=2^{-2},\ \ 
|\Omega_{2}^\pm\cap R_{1}| = |R_{2}|=2^{-3},\\ 
&  |\Omega_{k}^\pm\cap R_{k-1}| =2^{-(2k-2)} , |R_{k}|=2^{-(2k-1)} 
\ \ k=3,4,\dots.
\end{aligned}
  \]
In the case (C)     
$\Omega_{1}^{+}= \Omega_{1}^{-} =R_1, \ |R_1|=\frac{1}{2}$,
and the other densities in (7) are increased by a factor of $2$. 
\end{proposition} 
This proposition is sufficient to obtain the densities
for $\Pi_0, \Pi_1 $ and $\Pi_*$.
\begin{proof} {\it (Theorem  4 \& 5, first part)}
Since $\Omega_k^+\cap \Omega_k^- = R_{k}$, the values of the densities in
(7)    tell us
  that in the $k$-th table of Theorem 12, the partition of $R_{k-1}$
  into the four cells $\Omega_k^+\setminus\Omega_k^-,
  \Omega_k^-\setminus \Omega_k^+,R_k,Z_k$,
  is actually a partition into four subsets of equal density.
If that is the case for all $k$, then 
it leads to the equal splitting of the densities
of  $\Pi_0,\Pi_1$ and $\Pi_*$ inside of $R_0\setminus R_1$,
$R_1\setminus R_2$, $R_2\setminus R_3$, etc.,
and consequently to the trichotomy
\[
|\Pi_0|= |\Pi_1|= |\Pi_*|= \frac{1}{3}.
\]
In the cases (A) and (B) there is an exception in the table for $k=2$:
there are only two nonempty cells in $R_1$ with equal density, one is
$R_2$, the other $Z_2$. Hence in the calculation
of $|\Pi_*|$ we get
\[
|Z_1| =|R_1| =\frac{1}{4},\ |Z_2|= |R_2| =\frac{1}{8},\ \  |\Pi_*| = |Z_1|+|Z_2|+\frac{1}{3}|R_2| = \frac{5}{12}.
\]
In the case (C)     the exception is in the first table:
$|Z_1|=|R_1|=\frac{1}{2}$.
Now we get $|\Pi_*| = |Z_1|+\frac{1}{3}|R_1| = \frac{2}{3}$.
\end{proof}
%By Theorem 12 $\Pi_{*} = Z_{1}\cup Z_{2}\cup Z_3\cup \dots$.
 
The densities for the sets $\Pi_l, l\geq 2$, are obtained in a more
convoluted way involving the sets $R_k  \cap \Gamma_{k+s}$.
By the calculations of the orders of the Galois groups
in Appendix A we will have established that
for every fixed $k\geq 1$ these sets 
form dyadic sequences (over $s$). More precisely
\begin{proposition}
    If $q_0$ is generic then 
  \begin{equation}
    |\left(R_1\cup Z_1\right)\cap \Gamma_s| = 2^{-(s+1)},\
    |R_k  \cap \Gamma_{k-1+s}| = 2^{-(2k+s)}, 
      \ s\geq 0,\ k\geq 1.
  \end{equation}
In the non-generic case (C)  all densities in (8)
  are bigger by a factor of $2$.

  In the non-generic cases (A) and (B),
   the densities in the second equality of (8)
   are bigger by a factor of $2$.

   In the case (B) the first equality of (8) is replaced
   by $|\left(R_1\cup Z_1\right)\cap \Gamma_s|=2^{-s}$ for $s\geq 1$.   
\end{proposition}
Given Proposition 16 we proceed with the completion of the proof of
Theorems 4 and 5.
\begin{proof} {\it (Theorem  4 \& 5, second part)}
  We start with the generic case. 
  In the $k$-th step of the proof we want to establish
  that the sequence  of sets $\{Z_k\cap \Pi_{s+2}\}_{s\geq 0}$ is dyadic.

  %new version  
  For $k=1$ using  (8) and the finite additivity we
  arrive at the prime densities of the sets
  $|Z_1\cap \Gamma_{s}| =2^{-(s+2)},  s \geq 0$. Since this sequence of sets
  is nested it follows that the sequence 
$\{Z_1\cap \left(\Gamma_{s}\setminus \Gamma_{s+1}\right)\}_{s \geq 0}$
is dyadic.
We have $Z_1\cap \left(\Gamma_{s}\setminus \Gamma_{s+1}\right)=
Z_1\cap \{p\ | \ 2^{s+1}\ | | \ \widehat p\}$. 
By Theorem 13 we get that $Z_1\cap \{p\ | \ 2^{s+1}\ | | \ \widehat p\}
=Z_1\cap \Pi_{s+2}$, and the first step is accomplished.
%end of new version

%For $k=1$ 
%  the first equality of (8) says that
%  the sequence of sets $\{\left(R_1\cup Z_1\right)\cap \Gamma_s\}_{s\geq 0}$
%  is dyadic, and the second that the sequence
%  $\{R_1\cap \Gamma_{s}\}_{s \geq 0}$ is dyadic.
%It follows that the sequence
%$\{Z_1\cap \Gamma_{s}\}_{s \geq 0}$ is dyadic.
%Since it is nested also the sequence 
%$\{Z_1\cap \left(\Gamma_{s}\setminus \Gamma_{s+1}\right)\}_{s \geq 0}$
%is dyadic.
%We have $Z_1\cap \left(\Gamma_{s}\setminus \Gamma_{s+1}\right)=
%Z_1\cap \{p\ | \ 2^{s+1}\ | | \ \widehat p\}$. 
%By Theorem 13 we get that $Z_1\cap \{p\ | \ 2^{s+1}\ | | \ \widehat
%Z_1\cap \Pi_{s+2}$, and the first step is accomplished.

%new version

For $k=2$ the first crucial observation is that by Theorem 13
$R_1\cap \Gamma_1 = R_2\cup Z_2$. Hence by (8) we get the prime densities 
$|\left(R_2\cup Z_2\right)\cap \Gamma_{s+1}|= 2^{-(s+3)}$, and
$|R_2\cap \Gamma_{s+1}|=2^{-(s+4)},  s\geq 0$.
We conclude, as in step one, that the sequence
$\{Z_2\cap \left(\Gamma_{s+1}\setminus \Gamma_{s+2}\right)\}_{s \geq 0}$
is dyadic. Again by Theorem 13 
$Z_2\cap \left(\Gamma_{s+1}\setminus \Gamma_{s+2}\right)
=Z_2\cap \{p\ | \ 2^{s+2}\ | | \ \widehat p\}
=Z_2\cap \Pi_{s+2}$, and the second step is done.

%end of new version

%For $k=2$ the first crucial observation is that by Theorem 13
%$R_1\cap \Gamma_1 = R_2\cup Z_2$. Hence by (8) we get that the sequence of
%sets 
%$\{\left(R_2\cup Z_2\right)\cap \Gamma_{s+1}\}_{s\geq 0}$
%is dyadic. Since also $\{R_2\cap \Gamma_{s+1}\}_{s\geq 0}$
%is dyadic we conclude, as in step one, that the sequences
%$\{Z_2\cap \Gamma_{s+1}\}_{s\geq 0}$ and 
%$\{Z_2\cap \left(\Gamma_{s+1}\setminus \Gamma_{s+2}\right)\}_{s \geq 0}$
%are dyadic. Again by Theorem 13 
%$Z_2\cap \left(\Gamma_{s+1}\setminus \Gamma_{s+2}\right)
%=Z_2\cap \{p\ | \ 2^{s+2}\ | | \ \widehat p\}
%=Z_2\cap \Pi_{s+2}$, and the second step is done.

It is clear that this procedure can be then repeated for all $k\geq 3$.
Since by Theorem 12 $\Pi_*=Z_1\cup Z_2\cup\dots$ we conclude that
the sequence $\{\Pi_{s+2}\}_{s\geq 0}$ is dyadic.

In the non-generic case (C) $R_0 = R_1\cup Z_1$ and hence
$|R_1\cup Z_1|=1$, but the proof needs no modifications,
and we get the same conclusion.

The non-generic cases (A) and (B)
require modifications only in the first two steps
of this proof.

In case (A) we have that $\left(R_1\cup Z_1\right)\cap \Gamma_1 = R_1$
and hence $Z_1\cap \Gamma_1$ is empty. It follows that
$Z_1\subset \{p\ | \ 2\ | | \ \widehat p\}$ and by Theorem 15
$Z_1=Z_1\cap \{p\ | \ 2\ | | \ \widehat p\} = Z_1\cap\Pi_2$.
That is what we get in step one. In step two we observe that
there are only two nonempty cells in $R_1$, namely $R_1  = R_2\cup Z_2$.
The rest of the step two, and further steps, are the same as in the
generic case, delivering that $\{Z_k\cap \Pi_{s+2}\}_{s\geq 0}$ is dyadic
for $k\geq 2$. Hence $\{\left(\Pi_*\setminus Z_1\right)\cap \Pi_{s+2}\}_{s\geq 0}$
is dyadic.

Since in step one we have established that $Z_1\subset \Pi_2$ we conclude that
$\{\Pi_{s+2}\}_{s\geq 1}$ is dyadic and
\[
    |\Pi_2| =|Z_1|+\frac{1}{2}\left(|\Pi_*| - |Z_1|\right) = 
    \frac{1}{4}+\frac{1}{2}\left(\frac{5}{12}-\frac{1}{4} \right) = \frac{1}{3}.
    \]
    In case (B) we have that $\left(R_1\cup Z_1\right)\subset \Gamma_1$
    so in step one we establish that $Z_1\cap \Pi_2$ is empty
    and $\{Z_1\cap \Pi_{s+2}\}_{s\geq 1}$ is dyadic. In step two again
    $R_1  = R_2\cup Z_2$, but the rest is the same, delivering that
    the sequences $\{Z_k\cap \Pi_{s+2}\}_{s\geq 0}$ are dyadic for $k\geq 2$.
    We can conclude then that 
    $\{\left(\bigcup_{j\geq 1} Z_j\right)\cap \Pi_{s+2}\}_{s\geq 1} =
    \{\Pi_{s+2}\}_{s\geq 1}$ is dyadic and 
\[
    |\Pi_2| =\frac{1}{2}\left(|\Pi_*| - |Z_1|\right) = 
    \frac{1}{2}\left(\frac{5}{12}-\frac{1}{4} \right) = \frac{1}{12}.
    \]
  \end{proof}

Let us expound on the cases left out in the part (C) of
Theorem 5. We say that a rational  $q$ is {\it circular} if
there is a rational $w$ such that $q^2+w^2 = 4$. 
 For a circular $q$ let $z = q+iw$. We have $\sqrt{2z} =
 \sqrt{2+q}\left(1+\frac{2-q}{w}i\right)$ and
 $\sqrt{2iz} =
 \sqrt{2-w}\left(1+\frac{2+w}{q}i\right)$.
 It follows that  $q$ is primitive if and only if
 $ \sqrt{2z}\notin \mathbb Q(i)$, and  $w$ is primitive if
and only if  $\sqrt{i}\sqrt{2z} \notin \mathbb Q(i)$. 
Hence only one of the two can be non-primitive, i.e, to have a root.
We say that a circular $q$ is {\it circular primitive} if
both $q$ and $w$ are primitive.
It transpires also from these formulas that 
$\sqrt{2}\sqrt{2+q}\sqrt{2-w} \in \mathbb Q$,
so that a circular $q$ is circular primitive
if and only if $2+q$ and $2(2+q)$ are not rational squares. 

For a circular $q$ 
the partitions $\rho(q)$ and $\rho(w)$ are closely related. 
 It follows from Theorem 2 and the relation 
$C_2(q)=q^2-2 = -(w^2-2)=-C_2(w)$, that  
 \begin{equation}
 \begin{aligned}
 & \Pi_2(w)= \Pi_0(q)\cup\Pi_1(q),\ \ \ 
 \Pi_2(q) = \Pi_0(w)\cup\Pi_1(w),\\
 &\Pi_k(q) = \Pi_k(w), k = 3,4,\dots.
 \end{aligned}
 \end{equation}
 We will say that $\rho(w)$ is {\it the associate partition} of $\rho(q)$.
   
 Assuming circular primitivity of $q$  it is instructive to inspect
 the first table from the Introduction, for $q$ and for $w$.
 $R_0$ contains  only two non-empty cells, $R_1$ and $Z_1$. 
Since 
$\sqrt{2(2+q)(2-w)} \in \mathbb Q$, we have $2+q = ab^2$ and
$2-w = 2ac^2$, for rational $a,b,c$. By circular primitivity
we have that $a$ and $2a$ are not rational squares.
There are   four subsets in the table:
$D_{k,j} = \{p |\  (2|p) = (-1)^k, (a|p) = (-1)^j \}, \ k,j =0,1$,
but they are placed differently in the tables for $q$, and for $w$, namely
\[
\begin{aligned}
  &D_{1,1} \subset \Pi_2(q)\cap Z_1, \ \ \ D_{1,1} \subset
  \Pi_0(w)\cup \Pi_1(w),\\
  &D_{1,0} \subset \Pi_2(w)\cap Z_1, \ \ \ D_{1,0} \subset
  \Pi_0(q)\cup \Pi_1(q).
\end{aligned}
\]

The phenomenon of elements of a partition being lumped together
and permuted as above can have arbitrary depth.
Indeed,  
let circular $q_0$ be circular primitive, $q_0^2+w_0^2=4$.
Let $z_0=q_0+iw_0$, and for $N=2^k$ we consider  $2^{-(N-1)}z_0^N= 
C_{N}(q_0)+iw_0U_{N}(q_0) = q_k+iw_k$. Clearly $q_k$ is circular
    and it is not  primitive, hence $w_k$, being also
    circular, is primitive and not circular primitive. 
The partitions $\rho(q_k)$ and $\rho(w_k)$ are connected as follows. 
\[
 \begin{aligned}
 & \Pi_2(w_k)=\Pi_0(q_0)\cup\Pi_1(q_0)\cup\dots \cup\Pi_{k+1}(q_0),\ 
 \Pi_{k+2}(q_0) = \Pi_0(w_k)\cup\Pi_1(w_k),\\
 &\Pi_s(w_k) = \Pi_{k+s}(q_0), \ \ s = 3,4,\dots.
 \end{aligned}
 \]
Hence with increasing $k$ the set $\Pi_2(w_k)$ grows to include all primes,
while $|\Pi_0(w_k)\cup\Pi_1(w_k)|$ goes to zero. 
The circular primitive partition $\rho(q_0)$ is covered  by Theorem 5,
and it gives the following densities for
the primitive partition $\rho(w_k)$
  \[
  |\Pi_0(w_k)|= |\Pi_1(w_k)| =\frac{1}{3}\frac{1}{2^{k+1}},\ \
  |\Pi_2(w_k)| =1-\frac{1}{3}\frac{1}{2^{k-1}}, 
  \]
  and the sequence $\{\Pi_s(w_k)\}_{s\geq 3}$ is dyadic.
This requires the additional information that   
$|\Pi_0(w_k)|=|\Pi_1(w_k)|$. It follows from 
the general scheme by which densities were obtained: it guarantees
$|\Pi_0|=|\Pi_1|$ in any primitive partition.
Actually the general scheme can be applied directly to give densities
for $\rho(w_k)$, depending on $k$. However it is somewhat cumbersome, and
anyway it is instructive to bring in the core partition $\rho(q_0)$.

This completes the project of calculating the densities
in the partition $\rho(q)$ for all values of $q$.

\section{Kummerian-like field extensions}

We are going to consider some field extensions of $\mathbb Q$.
Let $L_n=\mathbb Q(q_n,\widehat q_n)$,
where $q_n$ is a root of $C_{2^n}(z)-q_0$ and
$\widehat q_n$  a root of $C_{2^n}(z)+q_0$. Further
let $M_n^{+}=\mathbb Q(q_{n+1},\widehat q_n) $ and
$M_n^{-}=\mathbb Q(q_n,\widehat q_{n+1})$,.
First we prove that the definitions are correct, namely that they  do not
depend on the choice of the roots.
\begin{theorem}
  $L_n = \mathbb Q(\sqrt{-\delta},\xi_{n-1}, q_n)=
  \mathbb Q(\sqrt{-\delta},\xi_{n-1},\widehat q_n)$,
  where $\xi_{n-1}$ is a root of $C_{2^{n-1}}(x)$, $n\geq 1$.

$M_n^+=\mathbb Q(\sqrt{-\delta},\xi_{n-1}, q_{n+1}) $ and
  $M_n^-=\mathbb Q(\sqrt{-\delta},\xi_{n-1},\widehat q_{n+1})$.
\end{theorem}
\begin{proof}
  We will use the commutative group $G= SO(2,\mathbb C)$ of complex
  $2\times 2$ matrices $A(a,b)=\frac{1}{2}
\left[ \begin{array}{cc} a &     -b\\
    b  &   a \end{array} \right]$, where $a^2+b^2 =4$.
We have  $tr  A(a,b)=a$, and $A(a,b)A(-a,b) = -I$.
Further by (2) we get
\begin{equation}
A(a,b)^{k} = A\left(C_{k}(a),bU_k(a)\right).
\end{equation}

Let us prove first that $\mathbb Q(\sqrt{-\delta},\xi_{n-1}, q_n)
\subset   L_n   $. 
To that end we observe that $q_0=C_{2^n}(q_n)=\left(C_{2^{n-1}}(q_n)\right)^2-2$,
and hence $C_{2^{n-1}}(q_n) =\sqrt{2+q_0}$. Similarly 
$C_{2^{n-1}}(\widehat q_n) =\sqrt{2-q_0}$. It follows that
$\sqrt{2+q_0}\sqrt{2-q_0} = \sqrt{-\delta}\in L_n$.

Further, for the chosen roots $q_n$ and
$\widehat q_n$, we introduce  matrices

$B_{+} =A\left(q_n, \frac{\sqrt{-\delta}}{U_{2^n}(q_n)}\right)$
and
$B_{-}=A\left(\widehat q_n,
\frac{\sqrt{-\delta}}{U_{2^n}(\widehat q_n)}\right)$.
To check that
$B_\pm \in G$ we use (F9).  Further   using (10) we get
$B_\pm^{2^n} = A\left(\pm q_0,\sqrt{-\delta}\right)$.

Let $E = 
B_+B_-$, and $\xi_{n-1}= tr\ E = q_n\widehat q_n
+\frac{\delta}{U_{2^n}(q_n)U_{2^n}(\widehat q_n) }$.
We have
\[
E^{2^n}=B_+^{2^n}B_-^{2^n} =
A(q_0,\sqrt{-\delta})A(-q_0,\sqrt{-\delta}) = -I.
\]
It follows that $C_{2^{n-1}}(\xi_{n-1})= tr \ E^{2^{n-1}} = 0$,
and the inclusion is proven. 

To prove 
$  L_n\subset   \mathbb Q(\sqrt{-\delta},\xi_{n-1}, q_n)$
we consider $F= A(\xi_{n-1}, \frac{2}{U_{2^{n-1}}(\xi_{n-1})})$.
We have $F^{2^{n-1}} = A(0,2)$, and it follows that $F^{2^{n}}=-I$.
  The trace of the matrix  $FB_+$ is a root of
  $C_{2^{n}}(z)+q_0=0$. Indeed 
$\left(FB_+\right)^{2^n} =F^{2^n}B_+^{2^n} =
  -A(q_0,\sqrt{-\delta})$.

  The second part of the theorem follows from the first
  if we note that $q_{n+1}^2 -2$ is a root of $C_{2^n}(z)-q_0$.
\end{proof}

It is clear that  the eigenvalue of $F$ is a root of unity
$\zeta_{n+1}$, $\zeta_{n+1}^{2^{n}}=-1$, such that
$\xi_{n-1} = 2\Re \zeta_{n+1}$.
Hence the traces of  $F^d$ for odd $d < 2^{n}$ give us 
$2^{n-1}$  different roots of $C_{2^{n-1}}(z)=0$.
Further the traces of $F^dB_+$ for odd $d < 2^n$
are  different roots of $C_{2^n}(z)+q_0=0$
(the remaining $2^{n-1}$ roots
are obtained by changing sign).
Indeed  $\left(F^dB_+\right)^{2^n} =F^{2^nd}B_+^{2^n} =
-A(q_0,\sqrt{-\delta})$.

However we do not need these facts 
to prove Theorem 17. We will need them though to obtain
the following important Corollary.
\begin{corollary}
$L_n$  is the splitting field of the polynomial
  $C_{2^{n+1}}(x)+2-q_0^2$.
  
$M_n^{\pm}$ is the splitting field of the polynomial 
  $\left(C_{2^{n+1}}(x)\pm q_0\right)\left(C_{2^{n}}(x)\mp q_0\right)$.
  \end{corollary}
\begin{proof} Clearly $C_{2^{n+1}}(x)+2-q_0^2 =
  \left(C_{2^{n}}(x)- q_0\right)\left(C_{2^{n}}(x)+q_0\right)$.
  The first part follows directly from Theorem 17 since
  in each of the three descriptions of the field $L_n$ 
  there are exactly two of the three roots $q_n,\widehat q_n,\xi_{n-1}$.
  Hence there can be no dependence on the missing root.

  This argument does not apply to the second part. We
  still need to prove that $M_n^\pm$ contains all the roots
  of $C_{2^{n+1}}(x)\mp q_0$. We 
  will achieve it by the aforementioned 
  application of the matrix  
  $F= A(\xi_{n-1}, \frac{2}{U_{2^{n-1}}(\xi_{n-1})})$.
  We have $F^{2^{n}}=-I$ and $F^{2^{n+1}}=I$.
  We introduce $B =
  A\left(q_{n+1}, \frac{\sqrt{-\delta}}{U_{2^{n+1}}(q_{n+1})}\right)$,
  so that $B^{2^{n+1}} = A\left( q_0,\sqrt{-\delta}\right)$.
     The traces of $F^kB$ for $k < 2^{n+1}$
are  different roots of $C_{2^{n+1}}(z)-q_0=0$.
Indeed
$\left(F^kB\right)^{2^{n+1}} =F^{2^{n+1}k}B^{2^{n+1}} = A(q_0,\sqrt{-\delta})$.
To prove that  all the $2^{n+1}$ traces are
different we invoke
the following  property of the group $SO(2,\mathbb C)$:
if two elements have the same trace then they are equal
or one is the inverse of the other.

Assume to the contrary that  $F^{k_1}B$ and $F^{k_2}B$
have the same trace for $k_1 < k_2 < 2^{n+1}$.  We get that 
$F^{k_1}B = F^{k_2}B$, or $F^{k_1}B = (F^{k_2}B)^{-1}$.
In the first case $F^{k_2-k_1}=I$ and in the second
$F^{k_2+k_1}B^{2}=I$. The first equality contradicts
$F^{2^n}=-I$, and raising the second to the power
$2^n$ we get $A(q_0,\sqrt{-\delta}) =\pm I$, which would
require $q_0=\pm 2$, the excluded values.
\end{proof}

Now we can formulate the  Corollary crucial for the proof
of Propositions 15 and 16. As shorthand we will say that a polynomial
{\it splits over $\mathbb F_p$} if it splits into linear factors over
$\mathbb F_p$.
\begin{corollary}
  \[
\begin{aligned}
  R_k= &\{p\ |\ \ C_{2^{k+1}}(x)+2-q_0^2 \ \text{splits over}\
  \mathbb F_p\},\ k\geq 1,\\
   \Omega_{k+1}^\pm\cap R_{k}
   =& 
   \{p\ | \ \left(C_{2^{k+1}}(x)\mp q_0\right)\left(C_{2^{k}}(x)\pm q_0\right)
  \ \text{splits  over}\ \mathbb F_p\},\ k\geq 1, \\
  R_{k}\cap \Gamma_{k+s} 
  = \{p\ |  \ &C_{2^{k+s}}(x)\left(C_{2^{k+1}}(x)+2-q_0^2\right)
  \ \text{splits  over}\
  \mathbb F_p\},  k\geq 1, s\geq 0,\\
\left(R_{1}\cup Z_1\right)\cap \Gamma_{s} 
  = &\{p\ | \ (x^2+\delta)C_{2^s}(x)
  \ \text{splits  over}\  \mathbb F_p\},\ \ s\geq 0.
\end{aligned}
\]
\end{corollary}
\begin{proof}
  The first two claims are a direct consequence of Corollary 18.
  To get the third  we apply Proposition 14.
  To get the fourth we need to observe further that by Theorem 12
  $R_{1}\cup Z_1 =\{p\ | \ (-\delta|p) = 1\}$.
 \end{proof}
 Corollary 19 allows  the application of the Frobenius 
  theorem to the sets $R_{k},  \Omega_{k+1}^\pm\cap R_{k}$
  and $R_{k}\cap\Gamma_{k+s}$, for any $k\geq 0, s\geq 0$.
  It requires the calculation of
  the orders of the respective Galois groups. This is done
  in Appendix A, and it delivers Propositions 15 and 16 of the previous
  Section.

\section{Lucas numbers}
Let us turn to a  $2\times 2$ matrix $A$ with integer trace $T$
and integer determinant $Q\neq 0$.
We have $A^k = L_{k}A-QL_{k-1}I$, where $L_k(T,Q), 
k\in \mathbb Z$, 
are the {\it Dickson polynomials of the second kind}.
Since the determinant of a matrix can be  rescaled to $1$,  
the Dickson polynomials can be recovered from the Chebyshev polynomials
by the following procedure. First we consider the homogenization
of the Chebyshev polynomials
$U_n(T,R)= R^{n-1}U_n(\frac{T}{R})$. Since the Chebyshev polynomials contain
terms of only odd powers, or only even powers, in the resulting homogeneous
polynomial $U_n(T,R)$ the variable $R$ appears in even powers alone.
Hence $U_n(T,\sqrt{Q})$ is a polynomial in $T$ and $Q$, and it is equal
to the Dickson polynomial $L_n(T,Q)$.

Similarly the trace of $A^n$ is also a polynomial in $T$ and $Q$,
and we denote it by $K_n(T,Q) = L_{n+1}(T,Q)-QL_{n-1}(T,Q)$,
it is the {\it Dickson polynomial of the first kind}.
Note that $L_n$ is of degree $n-1$ and
$K_n$ is of degree $n$.

All the properties of Chebyshev polynomials
can be translated into properties of Dickson polynomials.
However one needs to exercise some caution because the polynomials
$W_{2k+1}$ and $V_{2k+1}$ become polynomials in $T$ and $\sqrt{Q}$.
We will only write down the equivalents of  $(F1)$ through $(F7)$
as needed in the proofs below,
but we want to record the 
following striking identity, which is the equivalent of (F8).
Denoting the discriminant by  $D=T^2-4Q$ we have
for any odd $n$ 
\begin{equation}
  \begin{aligned}
    D^{\frac{n+1}{2}}L_{n}(T,Q) = K_{n}(D,-DQ),\\
   D^{\frac{n+1}{2}}K_{n}(T,Q) =L_{n}(D,-DQ).
  \end{aligned}
\end{equation}
An important aspect of our presentation is
that we refrain from  assuming that $T$ and $Q$ are coprime.

Let us denote by $\mathcal Z$ the set of integer pairs $(T,Q), 
Q\neq 0$.
We are going to study the sets of prime divisors
of the sequences $L_n(T,Q)$ and $K_n(T,Q)$. 
We have  $L_n(aT,a^2Q) = a^{n-1}L_n(T,Q),  K_n(aT,a^2Q) = a^{n}K_n(T,Q)$, 
which means that these divisors are essentially the same for
$(T,Q)$ and $(aT,a^2Q)$, with integer $a$.
We say that $(T,Q)$ is {\it similar to} $(T_1,Q_1)$ if $T^2Q_1 = T_1^2Q$.
Note that according to this definition $(T,Q)$ and $(-T,Q)$ are always
similar, while $(T,Q)$ and $(T,-Q)$ are not similar unless $T=0$.

Clearly the similarity
is an equivalence relation. The equivalence class of $(T,Q)$
can be naturally identified with a rational number $q\in \mathbb Q$
by the formula $q=\frac{T^2}{Q}-2$. The choice of this particular
$q$ is very convenient, which will transpire in the following.

We say that $(T,Q)\in \mathcal Z$ is
{\it simple} if for any prime divisor $p$ of the trace $T$ the determinant
$Q$ is not divisible by $p^2$. Every $(T,Q)$ is similar to  exactly two
simple values of the form $(\pm aP,aR)$ with coprime $R$ and $P$,
and $R$ and $a$.  It is obtained directly by simplifying the fraction  
$q+2=\frac{T^2}{Q}= \frac{aP^2}{R}$, with a squarefree  $a$.

For $(T,Q) \in \mathcal Z$ in the equivalence class of the simple
value $(aP,aR)$  
we define the subsets of odd primes
$\Pi_s= \Pi_s(T,Q) \subset \Pi$. An odd prime $p$, which is not a divisor of
$R$ or $a$, is in $\Pi_0$ if there is
an odd $n$   such that $p$ divides $L_{n}(aP,aR)$,
and it is in $\Pi_s, s= 1,2,\dots$, if there is an odd $n$ such that $p$
divides $K_{l}(aP,aR)$, for $l=2^{s-1}n$.

Finally we  assign the divisors of $a$ to $\Pi_1$,
because if $p$ divides $a$  then 
$a^{-\frac{p+1}{2}}K_{p}(aP,aR)$ is divisible by $p$.
At the same time  $a^{-l}K_{2l}(aP,aR)$ and
$a^{-k}L_{2k+1}(aP,aR)$ are not divisible by $p$ for any $l$ and $k$.

We remove all the
divisors of $R$ from consideration.
\begin{theorem}
  For any   $(T,Q)\in \mathcal Z$
  in the equivalence class of the simple value $(aP,aR)$ and
  $q=\frac{T^2}{Q}-2$
\[
\Pi_s(T,Q) = \Pi_s(q), s=0,1,2,\dots .
\]
\end{theorem}
\begin{proof}
  By (F4) and (F2) we have for $q=\frac{T^2-2Q}{Q}$
  \[
  \begin{aligned}
    &L_{2k}(T,Q)= TL_{k}(T^2-2Q,Q^2) =TQ^{k-1}L_k(q,1)= TQ^{k-1}U_{k}(q),\\
    &K_{2k}(T,Q)= K_{k}(T^2-2Q,Q^2) =Q^{k}K_k(q,1)= Q^{k}C_{k}(q).
    \end{aligned}
  \]
  The second line gives us immediately that $\Pi_s(T,Q) =\Pi_s(q)$
  for $s\geq 2$.
  
  Further 
  by (F4) and (F2),  using $TL_{n}=L_{n+1}+QL_{n-1}$,   we get
  \[
  \begin{aligned}
  &TL_{2k-1}(T,Q)=L_{2k}(T,Q)+ QL_{2k-2}(T,Q)=TQ^{k-1}(U_{k}(q)
+  U_{k-1}(q))\\
&=TQ^{k-1}W_{2k-1}(q), \ \ \text{i.e.,} \ \ \ \ 
 L_{2k-1}(T,Q)=Q^{k-1}W_{2k-1}(q).
\end{aligned}
  \]
It means that $\Pi_0(T,Q) =\Pi_0(q)$.  
Similarly
\[
  \begin{aligned}
  &K_{2k-1}(T,Q)=L_{2k}(T,Q)- QL_{2k-2}(T,Q)=TQ^{k-1}(U_{k}(q)-
  U_{k-1}(q))
  \\
  &=TQ^{k-1}V_{2k-1}(q).
\end{aligned}
  \]
  We get  $\Pi_1(T,Q) =\Pi_1(aP,aR)=\Pi_1(q)$,
  except for the status of the divisors
  of $aP$. We still need to check that they belong to 
  $\Pi_1(q)$. For any odd prime divisor $p$ of $aP$
  we have $q=\frac{aP^2}{R}-2 = -2 \ mod \ p$.
    By the rules of modular arithmetic,
if $q =-2 \ mod \ p$ then $V_{p}(q) = V_{p}(-2) \ mod \ p$.
Now we conclude with $V_{p}(-2) = W_p(2)=p   = 0 \ mod  \ p$.
  \end{proof}
This theorem allows immediate translation of the results of Section 4
into the language of Lucas  sequences.
In the rest of this Section we will perform such a translation 
for some chosen claims. 

We introduce the map $\psi: \mathcal Z \to \mathcal Z$ by the formula
$\psi(T,Q) = (T^2-2Q,Q^2)$. It is clear that $\psi$ is well
defined on the similarity
equivalence classes in $\mathcal Z$, and since $\psi(q) = q^2-2$,
it corresponds to the squaring of the partitions.
Further we can naturally introduce roots of some elements in  $\mathcal Z$.
It is not hard to see that $(T,Q) \in \mathcal Z$ has a root if and only
if $Q=R^2$, and then there are always two roots $(2R+T,(2R+T)R)$ and
$(2R-T,(2R-T)R)$. They coincide only in the trivial case $T=0$.
In particular all the values $(T,1)$ have two roots $(2+T,2+T)$ and
$(2-T,2-T)$, while  the values $(T,-1)$ do not have roots.

In the same vein we introduce the mapping
$\mathcal B: \mathcal Z \to \mathcal Z$ by the formula
$\mathcal B(T,Q) = (D,-DQ)$, if $D = T^2-4Q\neq 0$.
If $D=0$ then $\mathcal B(T,Q) = (0,Q)$.
It is clear that  $\mathcal B$ is well defined on the similarity
equivalence classes in $\mathcal Z$,
 and  since $\mathcal B(q) =-q$, it is involutive, i.e.,
  $B\circ B = Id$.
The value $\mathcal B(T,Q)$ is the {\it twin } of $(T,Q)$.
The defining property of the twins is that squares of twins coincide,  i.e.,
$\psi\circ \mathcal B =\psi$.
Let us note that every equivalence class  has a unique
twin, and simple values $(aP,aR)$ and $(bS,bR)$ are twins
if and only if $aP^2+bS^2=4R$.

For twin simple values
$(aP,aR)$ and $(bS,bR)$ we have by Theorem 2 
\[
\Pi_j(aP,aR) = \Pi_{1-j}(bS,bR), j=0,1, \Pi_s(aP,aR) = \Pi_s(bS,bR), s\geq 2.
\]
The first property can be also obtained directly
from the relation (11) between the polynomials $L_n$ and $K_n$ for odd $n$.

Further $\mathcal B\left(\psi (aP,aR)\right)=
\mathcal B\left(\psi(bS,bR)\right)
=\left(abPS,abR^2\right)$. Indeed $\psi(aP,aR)=(aP^2-2R,R^2)$
and its discriminant
$D= aP^2(aP^2-4R)= abP^2S^2$.  
It follows from Corollary 3  that 
\[
    \begin{aligned}
      &\Pi_0\left(abPS,abR^2\right) = \Pi_2(aP,aR) = \Pi_2(bS,bR),\\ 
      &\Pi_{1}\left(abPS,abR^2)\right)
      = \Pi_{0}(aP,aR)\cup\Pi_{1}(aP,aR) =\Pi_{0}(bS,bR)\cup\Pi_{1}(bS,bR),\\
      &\Pi_k\left(abPS,abR^2\right) = \Pi_{k+1}(aP,aR)=\Pi_{k+1}(bS,bR),
      k\geq 2.
        \end{aligned}
\]

The case of $a=b$ is possible only with $a=b=1$ or
$a=b=2$, because $a$ and  $R$ are coprime.
This case was studied in great detail in the paper of
Ballot,\cite{B3}. In particular he showed that for almost all values of
$(P,R)$ in this class 
the density $|\Pi_0(P,R)| ={1\over 6}$.

It corresponds to a point $z=q+iw$ on the  circle $q^2+w^2=4$,
with both $q$ and $w$ rational. 
It  was the case of circular $q$ studied at the end of Section 4.
One such point gives rise to
four partitions $\rho(\pm q), \rho(\pm w)$ which are associates of each other
as described in (9).
To   translate it into the present
language let us observe that 
if $a=b$ then the value $\mathcal B\left(\psi(aP,aR)\right)
=(aPS,R^2) $, and hence it has two roots which are twins.
\[
        \begin{aligned}
&(2R+aPS,(2R+aPS)R) = (a(P+S),2aR),\\
&(2R-aPS,(2R-aPS)R) = (a(P-S),2aR).
        \end{aligned}
        \]
        Hence the associate partition of a circular $\rho(aP,aR)$ (or the twin
        $\rho(aS,aR)$) is the partition $\rho(a(P\pm S),2aR)$. 
        Now (9) reads as         
 \[
 \begin{aligned}
 & \Pi_2(P\pm S,2R)= \Pi_0(P,R)\cup\Pi_1(P,R)= \Pi_0(S,R)\cup\Pi_1(S,R),\\  
 &\Pi_2(P,R) = \Pi_0(P\pm S,2R)\cup\Pi_1(P\pm S,2R),\\
 &\Pi_k(P,R) = \Pi_k(S,R)= \Pi_k(P\pm S,2R), k = 3,4,\dots.
   \end{aligned}
 \]

$(T,Q)\in \mathcal Z$ in the similarity equivalence class
$q=\frac{T^2}{Q}-2$ is {\it primitive} if the partition $\rho(q)$
is primitive, i.e. neither $(T,Q)$ nor its twin $(D,-DQ)$  
has a root. Hence  $(T,Q)$ is primitive if and only if neither $Q$ nor
$-DQ$ is a rational square.

Let us translate the table from the Introduction
into the language of $(T,Q)$ values.
We have $2+q = \frac{T^2}{Q}, \ 2-q= \frac{-D}{Q}$
and $4-q^2=-D\frac{T^2}{Q^2}$.
Using the language of simple twins $(aP,aR), (bS,bR), aP^2+bS^2=4R,$
we get $2+q = \frac{aP^2}{R}, \ 2-q= \frac{bS^2}{R}$
and $4-q^2=ab\frac{P^2S^2}{R^2}$. 

Assuming primitivity of the partition, i.e, 
$Q$ and $-DQ$ are not rational squares,
or equivalently assuming that $aR$ and $bR$ are not rational squares,
we get the following table.
\begin{center}
\begin{tabular}{|c|c|c|}
  \hline
  & $(Q|p)=1 $ & $(Q|p)=-1 $ \\
  \hline
  $ (-D|p)=1 $ &
  \begin{tabular}{c|c}
    $(2|p) = 1 $ &  $(2|p) = -1 $ \\
  \hline
  ??? & $ \Pi_0 \cup \Pi_1$  
  \end{tabular}
  & \begin{tabular}{c|c}
   $(2|p)=1 $ & $(2|p)=-1 $ \\
  \hline
  $\Pi_*\setminus\Pi_2$ & $\Pi_2$  
\end{tabular} \\
  \hline
  $(-D|p) =-1$ & $\Pi_0$ & $\Pi_1$\\
  \hline
\end{tabular}
\end{center}
Parts of the contents of the last  table
can be found in the literature, from the old paper
of Sierpinski,\cite{Si}, the book of Ribenboim,\cite{R},
to the recent paper of Somer and Krizek, \cite{S-K}.
In particular for 
$(T,Q) =(1,-1)$ the information from the table is provided
in the paper of Moree,\cite{M3}(page 280).

The translation of the further refinements of the table, by
Theorem 12, is  possible, but not particularly illuminating.

Let us finally translate the conditions for different densities
from Theorems 4  and 5.

$(T,Q)$ is generic if
and only if $Q,-D,-DQ,2Q,-2D,-2QD$ are not rational squares.

A primitive $(T,Q)$ is non-generic of type (A)
if and only if $2Q$ or $-2DQ$ is a rational square, and $-D$ is not
a rational square.

A primitive $(T,Q)$ is non-generic of type (B)
if and only if $-2D$  is a rational square.

A primitive $(T,Q)$ is non-generic of type (C)
if and only if $- D$ is a rational square and $2Q$ is not
a rational square.

These conditions are mutually exclusive,
and exhaustive for primitive partitions,
with the exception of circular non-primitivity in case (C) as explained
in Section 4. For example, for  $(T,Q) =(1,-2)$,
of the six conditions defining genericity, only the last one
is violated, $-2DQ=6^2$. It is non-generic of type (A).
The twin value is $(3,2)$, and here $2Q =2^2$.
Twins have always the same genericity status.
The value  $(T,Q) =(2,3)$ is primitive non-generic of type (B),
because $-2D = 4^2$. The twin value is $(2\cdot 2,2\cdot 3)$.

In the non-generic case (A), if for example  $Q=2X^2$ for a natural $X$,
the table  becomes
\begin{center}
\begin{tabular}{|c|c|c|}
  \hline
  & $(2|p)=1 $ & $(2|p)=-1 $ \\
  \hline
 $(-D|p) =1$ & ??? & $\Pi_2$  
\\
  \hline
  $(-D|p) =-1$ & $\Pi_0$ & $\Pi_1$\\
  \hline
\end{tabular}
\end{center}
In the non-generic case (B), i.e., if 
 $-D =2X^2$, for a natural $X$,  then the table  becomes
\begin{center}
\begin{tabular}{|c|c|c|}
  \hline
  & $(Q|p)=1 $ & $(Q|p)=-1 $ \\
  \hline
 $(2|p) =1$ & ??? & $\Pi_*\setminus\Pi_2$  
\\
  \hline
  $(2|p) =-1$ & $\Pi_0$ & $\Pi_1$\\
  \hline
\end{tabular}
\end{center}

The values of type (B) can be  enumerated
by the equation $P^2+2S^2 =4R$ with coprime $R$ and $P$, and
coprime $R$ and $2S$.

Finally let us mention that Lehmer sequences 
appear 
each time a simple value $(aP,aR)$ has $a\neq \pm 1$. It may be that
the twin value is Lucas, i.e., $b=\pm 1, (\pm S,\pm R), aP^2\pm S^2=4R$,
but that is rare. On the other hand the twin of a Lucas sequence
is a Lucas sequence only in the cases $P^2\pm S^2=4R$.

\section{Regular and chaotic  dynamics}

Let us consider the circle
$\mathbb S^1=\{z=q+iw\ | \ q^2+w^2=4\}$ of radius $2$
and the family of rotations $\Phi_{z_1}:\mathbb S^1\to\mathbb S^1, \
 \Phi_{z_1}(z)=\frac{1}{2}z_1z$. Rotations are prototypical
examples of regular dynamical systems. If $z_1=q_1+iw_1$ has a rational
real part then all points in the orbit $\Phi_{z_1}^{n}(2) =  q_n+iw_n$
have rational real parts $q_n$. Indeed $q_n = C_{n}(q_1)$ and
$w_n=w_1U_{n}(q_1)$. If $q_1=\frac{a_1}{b}$, with coprime integers
$a_1$ and $b$, then $q_n=\frac{a_n}{b^n}$ for some integers
$a_n$, coprime with $b$. The divisors of the integers
$a_n, n = 1,2,\dots$, form
the subset $\Pi_*(q_1)$ and hence by Theorem 4
 it has prime density equal in the generic case
to $1/3$. The primitivity condition translates as
$\sqrt{2z_1}$ and $\sqrt{-2z_1}$ having irrational real parts.
The genericity has four more conditions: 
$\sqrt{z_1}$ and $\sqrt{-z_1}$ have irrational real parts,
and $w_1$ and $\sqrt{2w^2_1}$ are irrational.

The numerators $a_n$ for odd and even $n$
have no common divisors. Their divisors split into $\Pi_2$ and
$\Pi_*\setminus\Pi_2$. In the generic case it is an even split,
but not so in the non-generic cases (A) and (B). 

If the imaginary part $w_1$ is also rational then the density
jumps to    $2/3$, as long as $\sqrt{z_1}$ and $\sqrt{2z_1}$ 
(or equivalently $\sqrt{2z_1}$ and $\sqrt{2iz_1}$) have irrational real parts.
That is the  case covered in part (C) of Theorem 5.

For the  imaginary parts $w_n=w_1U_n(q_1)$ all odd primes are  divisors of
some of the
numerators of $U_n(q_1)$, except for the  divisors of the denominator
$b$.

Clearly we need to exclude the cases of  rational (i.e.,
periodic) rotations. It happens only  for  $q_1=0,\pm 1, \pm 2$.

Let us now turn to another prototypical example in dynamics,
the map $\Psi:\mathbb S^1\to\mathbb S^1, \
\Psi(z)=\frac{1}{2}z^2$. If $z_0 = q_0+iw_0$ then $\Psi(z)=q_0^2-2+iw_0q_0$,
and the orbit $\Psi^n(z_0)= q_n+iw_n= C_{2^n}(q_0)+iw_0U_{2^n}(q_0)$.
  The projection of this  map on the real axis is the special case of
  an interval map in the
  quadratic family, $\psi:[-2,2]\to [-2,2], \ \psi(q) = q^2-2$.
    These systems have good statistical properties, for example  they are
    naturally     equivalent to coin tossing.

If $q_0=\frac{a_0}{b}$ with coprime integers $a_0$ and $b$,
    then  all the points $q_n=\psi^n(q_0)$
    in the orbit of the
    quadratic map $\psi$ are rational $q_n = \frac{a_n}{b^{2^n}}$.
   We need  to exclude the special values $q_0=0,\pm1,\pm 2$ which are the only
   rational     periodic, or pre-periodic points of $\psi$.
   With these exceptions, by Theorem 1, the numerators
   $a_n, n= 0,1,\dots$ are coprime,
     except for the factor $2$.
    By Corollary 10 their odd prime divisors cannot be small,
    namely if $p$ divides $a_n$ then $p \geq 2^{n+1}-1$.

    The set  $\mathcal A$ of the divisors of the
    numerators has prime density zero.
    We first prove using the fact that the sets in the partition
    $\rho(q_0)$ do have prime densities.
    Indeed let us observe that
    $
    \mathcal A =     \left(\mathcal A\cap\bigcup_{k=0}^N\Pi_k(q_0)\right) \cup
    \left(\mathcal A\cap\bigcup_{k=N+1}^{\infty}\Pi_k(q_0)\right).
    $
    The first set is finite and the second, being
    a subset of a set with a small prime density,
    has small upper prime density. Increasing $N$ we obtain zero prime
    density for $\mathcal A$.

    For the imaginary parts of the points in the $\Psi$-orbit, 
    $w_n=w_0U_{2^n}(q_0)$ the divisors of
    the numerators of $U_{2^n}(q_0)$ are the same as the divisors
    of the real parts. It follows from the factorization
    $
  U_{2^{n}} = C_{1}C_{2}C_{4}\dots C_{2^{n-1}}$.

We can repeat this translation also in the case of higher 
powers, namely for $m\geq 2$ we define the mappings    
$\Psi_m:\mathbb S^1\to\mathbb S^1, \
\Psi_m(z)=\frac{1}{2^{m-1}}z^m$.
If $z_0 = q_0+iw_0$ then $\Psi_m(z)=C_m(q_0)+iw_0U_m(q_0)$,
and the orbit $\Psi_m^n(z_0)= q_n+iw_n= C_{m^n}(q_0)+iw_0U_{m^n}(q_0)$.
  The projection of this  map on the real axis is an interval map defined by
  the respective Chebyshev polynomial
  $\psi_m:[-2,2]\to [-2,2], \ \psi_m(q) = C_m(q)$.
  If $q_0=\frac{a_0}{b}$ with coprime integers $a_0$ and $b$,
    then the all the points $q_n=\psi^n(q_0)$
    in the orbit of the
    $\psi_m$ are rational $q_n = \frac{a_n}{b^{m^n}}$.
    The numerators $a_n, n= 0,1,\dots$ are coprime
    only for even $m$. 
    For odd $m$ the sets of odd prime divisors
    form an ascending chain of subsets of $\Pi_2(q_0)$.
    
The prime density
of the set $\mathcal A$ of prime divisors of the numerators $a_n$
is always zero. For even
$m$ the  argument above is applicable. For odd $m$ we need to invoke the
following property of the  index of appearance $\xi(p)$: for any odd prime
$p$ we have  $p=(\delta|p)  \ mod \ \xi(p)$. It follows directly from
Theorem 7 and Corollary 10.

Let us further introduce
$\mathcal A^N=\{p\in \mathcal A \ | \ a_n \neq  0 \ mod \ p \ \ \text{for } \ \
n \leq N\}$. We claim that for large $N$ the subset $\mathcal A^N$ has small
upper density. Indeed let $m=r_1^{s_1}r_2^{s_2}\dots r_k^{s_k}$ be the
prime decomposition of $m$. For any $p\in \mathcal A^N$ there is
at least one prime divisor $r_j$ such that $\xi(p) = lr_j^{Ns_j}$ for some
natural $l$. Hence  $p=\pm 1 \ mod \   r_j^{Ns_j}$, and the density of
the set of all primes that satisfy the last congruence is given by the
Dirichlet theorem as $2\left(\varphi(r_j^{Ns_j})\right)^{-1}$,
where $\varphi$ is the Euler function. Taking the union
over the $k$ prime divisors of $m$ we still get an arbitrarily small
upper prime density of $\mathcal A^N$,
if only $N$ is sufficiently large. Since $\mathcal A \setminus
\mathcal A^N$ is finite we get our claim.

Incidentally this argument works also in the case of even $m$.

\section{Appendix A:  Towers of quadratic extensions }

Let $K_0$ be a field of characteristic $0$ and $a_0\in \mathbb Q,
a_0\neq \pm 2$.
  We define $K_1= K_0(\sqrt{4-a_0^2})$,
  and $K_2= K_1(a_1)$ where $a_1= \sqrt{2+a_0}$. We have that
  $\widetilde a_1 =\sqrt{2-a_0} \in K_2$.
  Further, by induction, given a field
  $K_r=K_{r-1}(a_{r-1}), a_{r-1}=\sqrt{2+a_{r-2}},
  \widetilde a_{r-1}=\sqrt{2-a_{r-2}},
  r\geq 2$, we define
  $K_{r+1}=K_r(a_r), a_r=\sqrt{2+a_{r-1}}$ and   $\widetilde a_r
  =\sqrt{2-a_{r-1}}$.

  Let us observe first that $a_{r}\widetilde a_{r} =
  \widetilde a_{r-1} \neq 0$. Indeed we have  
  \[
    a_{r}\widetilde a_{r} =\sqrt{2+a_{r-1}} \sqrt{2-a_{r-1}}=\sqrt{4-a_{r-1}^2}
    = \sqrt{2-a_{r-2}} =  \widetilde a_{r-1}.
    \]
  It follows by induction that     $a_{r}\widetilde a_{r}\neq 0$ and 
  $\widetilde a_{r}\in K_{r+1}=K_r(a_r)$. Consequently
  $K_r(\widetilde a_r)=K_r(a_r)$. The following Lemma is essentially
  our only tool, that will be used repeatedly.
\begin{lemma}
    If $[K_{s+1}:K_s] =2$, for some $s\geq 1$,
    then     $[K_{r+1}:K_{r}] =2$ for every $r\geq s+1$.
  \end{lemma}
  \begin{proof}
    Suppose to the contrary that $n\geq 1$ is the first natural number such
    that $K_{n+2}=K_{n+1}$. It means that $a_{n+1}=\sqrt{2+a_n}
    \in K_{n+1}= K_n(a_n)$.
    Hence there are $\alpha,\beta \in K_n$ such that
    $2+a_n = (\alpha a_n +\beta)^2$. It follows that 
    $\alpha^2a^2_n+\beta^2 =2$, $2\alpha\beta =1$.
    Note that while $a_n\notin K_n$ by the inductive
    assumption, $a_n^2\in K_n$     by definition.
    Combining the two equations we arrive at
    $4(\beta^2-1)^2 = 4-a_n^2$. Hence $\sqrt{2-a_{n-1}} 
    =\sqrt{4-a_n^2} = 2(\beta^2-1)$. We arrived at a contradiction
    with the inductive assumption that $\sqrt{2+a_{n-1}}$, and hence also 
    $\sqrt{2-a_{n-1}}$, are not in the field $K_n$. It proves that
    actually $\sqrt{2+a_n} \notin K_{n+1}$.
  \end{proof}
  In the first application of Lemma 21 we put 
  $K_0= \mathbb Q(\sqrt{b}),  a_0=0, K_1=K_0,
  K_2=\mathbb Q(\sqrt{b},\sqrt{2}) $. 
  If $b$ and $2b$ are  not  rational squares
  then $[K_2:K_1]=2$, and hence $[\mathbb Q(\sqrt{b},\xi_k):\mathbb Q]
  = 2^{k+1}$, where $\xi_k$ is a root of the irreducible polynomial $C_{2^k}(x)$.
  
For two  rationals $a,b$ we say that they are {\it a generic pair}
if $a,b,  c=ab $ and $2a,2b,2c$ are not rational squares.
\begin{lemma}
   $a,b$ is a generic pair if and only if
    $[\mathbb Q(\sqrt{a},\sqrt{b},\sqrt{2}):\mathbb Q] =8$.
  \end{lemma}
\begin{proposition}
  If    $a,b$ is a generic pair then 
  $[\mathbb Q(\sqrt{a},\sqrt{b},\xi_k):
    \mathbb Q] = 2^{k+2}$.
  \end{proposition}
\begin{proof}
We apply Lemma 21 choosing 
$K_0=\mathbb Q(\sqrt{a},\sqrt{b})$ and $a_0=0$ so that
$K_1=K_0$ and
  $K_2=\mathbb Q(\sqrt{a},\sqrt{b},\sqrt{2})$.
Hence by Lemma 22  $[K_2:K_1]=2$, and the first
equality follows.
\end{proof}

We say that a rational $q_0$ is {\it generic}
if the pair $2+q_0,2-q_0$ is generic. For given $q_0$
we define $\delta = q_0^2-4$, and $q_n$ is any root of
    the polynomial $C_{2^n}(x)-q_0$.
  \begin{proposition}
    For a generic $q_0$, 
    \[
[\mathbb Q(\sqrt{-\delta},\xi_k, q_n):
      \mathbb Q] = 2^{k+n+1}.
    \]
    \end{proposition}
\begin{proof}
  In Lemma 21 we put $K_0=\mathbb Q(\xi_k),K_1=K_0(\sqrt{-\delta})$,

  $K_2= K_0(\sqrt{-\delta},\sqrt{2+q_0})$.
  If $2+q_0$
 and $2-q_0= \frac{-\delta}{2+q_0}$ are a generic pair then
 $[K_2:\mathbb Q]=2^{k+2}$ by Proposition 23. 
  Since 
$[\mathbb Q(\xi_k):
    \mathbb Q] = 2^{k}$
  we must conclude that $[K_2:K_1]=2$ and
  Lemma 21 applies.
\end{proof}

We consider now the non generic case of $-\delta =w^2$
for some rational $w$, with
$2+q_0$ and $2(2+q_0)$ not rational squares.
We will say in such a case that $q_0$ is
{\it circular primitive}.
  \begin{proposition}
    For a circular primitive  $q_0$ we have for $n\geq 1, k\geq 1$,
    \[
      [\mathbb Q(\xi_k,q_n): \mathbb Q] = 2^{n+k},
    \]
    \end{proposition}
  \begin{proof}
First we establish the degree for $n=1$.  
In Lemma 21 we choose $K_0=\mathbb Q(\sqrt{2+q_0})$ and $a_0=0$.
We have $K_1=K_0, 
K_2= K_0(\sqrt{2}) = Q(\sqrt{2+q_0},\sqrt{2})$. Hence $[K_2:K_1]=2$ and
it follows that $      [\mathbb Q(\sqrt{2+q_0},\xi_k): \mathbb Q] = 2^{k+1}$.

Applying Lemma 21 again we put 
$K_0=\mathbb Q(\xi_k), K_1=K_0(\sqrt{-\delta})= K_0 ,
  K_2= K_0(\sqrt{2+q_0})$. 
 We have  established
  that $[K_2:\mathbb Q]=2^{k+1}$. 
  Since 
$[\mathbb Q(\xi_k):
    \mathbb Q] = 2^{k}$ we must conclude that $[K_2:K_1]=2$ and
  Lemma 21 applies. 
\end{proof}
In preparation for the non-generic case (A) we need the following 
\begin{lemma}
  For a rational  $c$, if $2-c^2$  and
  $2(2-c^2)$ are not  rational squares then  
  $[\mathbb Q(\sqrt{2+\sqrt{2}},\sqrt{2-c^2},
    \sqrt{2+c\sqrt{2}}):\mathbb Q] =16$.
\end{lemma}
\begin{proof}
  It is not hard to establish that under the assumptions
  
  $[\mathbb Q(\sqrt{2+\sqrt{2}},\sqrt{2-c^2}):\mathbb Q] =8$.
To finish the proof let us assume to the contrary that 
$2+c\sqrt{2} = (d_0+d_1\sqrt{2-c^2})^2, d_0,d_1\in \mathbb Q(\xi_2)$.
It follows that by necessity  $d_0d_1=0$. In both cases we are led
to the equation
$b(2+c\sqrt{2}) = (h_0+h_1\xi_2)^2, h_0,h_1\in \mathbb Q(\sqrt{2})$,
for some  rational $b$. Again $h_0h_1=0$ but this time the two cases
are slightly different. We will do only the case $h_0=0$, the other
case is left to the reader.

We obtain the following equation 
$b(2+c\sqrt{2}) = (g_0+g_1\sqrt{2})^2(2+\sqrt2)$ for rational $b, g_0$
and $g_1$. We obtain further
\[
b\left(2-c +(c-1)\sqrt{2}\right) = g_0^2+2g_1^2 +2g_0g_1\sqrt{2}.
\]
Introducing $t= \frac{g_0}{g_1}, s = \frac{2-c}{c-1}$ we get
the quadratic equation $t^2-2st+2=0$, which has rational solutions only
if $(s^2-2)(c-1)^2 = 2-c^2$ is a rational square.
\end{proof}

We are ready to deal with the exceptional cases where $-2\delta$  or $2(2+q_0)$
is a rational square.

  \begin{proposition}
    If $-\delta$ and  $2-q_0$ are not rational squares,
    but $2(2+q_0)$ is a rational square,
or if $2+q_0$ and  $2-q_0$ are not rational squares,
    but $-2\delta$ is a rational square, then 
\[
  [\mathbb Q(\sqrt{-\delta},q_1): \mathbb Q] = 4,\ \ 
      [\mathbb Q(\sqrt{-\delta},\xi_k,q_n): \mathbb Q] = 2^{n+k},
      \]
      for   $k\geq 1,n\geq 1$.
  \end{proposition}
  Note that 
  $-2\delta$ and $2(2-q_0)$ under the first set of assumptions,
  and  $2(2+q_0)$ and  $2(2-q_0)$ under the second set,
  cannot be rational squares.
  \begin{proof}
The first claim is a straightforward conclusion from the assumptions.

    For $n=1$ in the second claim
    we apply Lemma 21 with 
$K_0=\mathbb Q(\sqrt{-\delta},q_1)$ and $a_0=0$, so that
$K_1=K_0$ and
$K_2=\mathbb Q(\sqrt{-\delta},q_1,\sqrt{2}) =K_0$.
However $K_3=\mathbb Q\left(\sqrt{-\delta},q_1,\sqrt{2+\sqrt{2}}\right)$,
and by Proposition 23 $[K_3:K_2]=2$. It follows that
$[\mathbb Q(\sqrt{-\delta},q_1,\xi_k): \mathbb Q] = 2^{k+1}$.
Note that so far we have done simultaneously both sets of assumptions.

Now for $n\geq 2$ we apply again Lemma 21 with 
$K_0=\mathbb Q(\xi_k),k\geq 1, K_1=K_0(\sqrt{-\delta}),
K_2= K_1(\sqrt{2+q_0})$. If $-2\delta$ is a rational square then
$K_1=K_0$, but $[K_2:K_1]=2$ (because $2+q_0$ and $2(2+q_0)$ are not
rational squares). This gives us the claim.

If $2+q_0 = 2c^2$ then $q_1 =c\sqrt{2}, -\delta=4c^2(2-c^2) $
and $q_2=\sqrt{2+c\sqrt{2}}$.
We get $K_2=K_1=K_0\left(\sqrt{2-c^2}\right)$ and
$K_3=K_2(q_2)= K_0\left(\sqrt{2-c^2}, \sqrt{2+c\sqrt{2}}\right)$.
We claim that
$[K_3:\mathbb Q] = [\mathbb Q\left(\sqrt{2-c^2},
    \sqrt{2+c\sqrt{2}},\xi_k\right):\mathbb Q] =2^{k+2}$.  
Once this is established we conclude that $[K_3:K_2]=2$
(because clearly $[K_2:\mathbb Q]=[\mathbb Q\left(\sqrt{2-c^2},
  \xi_k\right):\mathbb Q] =2^{k+1}$),
  and Lemma 21 can be applied.

  To establish that 
  $[\mathbb Q\left(\sqrt{2-c^2},\sqrt{2+c\sqrt{2}},\xi_k\right):
    \mathbb Q] =2^{k+2}$
  we employ again Lemma  21 with $a_0=0, K_0=
  \mathbb Q\left(\sqrt{2-c^2},\sqrt{2+c\sqrt{2}}\right),
          [K_0:\mathbb Q] =8$.
  Now $K_2=K_1=K_0$, however $K_3=
  \mathbb Q\left(\sqrt{2-c^2},\sqrt{2+c\sqrt{2}},\sqrt{2+ \sqrt{2}}\right)$,
  and by Lemma 26 $[K_3:K_2]=2$, which ends the proof.
\end{proof}

\section{Appendix B:
  Chebyshev polynomials, and their algebraic properties}

 The nine identities, (F1) through (F9), involving Chebyshev polynomials  
are proven here in three propositions, because of
 the affinities in the proofs.
 \begin{proposition}
For any integers  $n$ and $k$
 \[
   \begin{aligned}
     (F1)   \ \ \  \ \ \ \ \ \ \
     &C_{nm}= C_m\left(C_n\right),\\
     (F2)  \ \ \  \ \ \ \ \ \ \
     &U_{nm}=U_{n}U_{m}\left(C_n\right).
\end{aligned}
   \]
\end{proposition}
\begin{proof}
  (F1) follows immediately
  from the interpretation of $C_n(q)$ as the trace of the $n$-th
  power of a matrix
  with trace $q$. Indeed $C_{nm}= tr  A^{nm}= tr\left(A^{n}\right)^m$.

  Using  (1) three times we get  
 \[
 U_{nm}A-U_{nm-1}I =\left(A^{n}\right)^m =U_m(C_n) A^{n}-U_{m-1}(C_n)I
 =U_m(C_n)U_n A - \gamma I,
 \]
 where $\gamma=U_{m}(C_n)U_{n-1} +U_{m-1}(C_n)$.
 Comparing the beginning and the end we obtain (F2). 
   \end{proof}

\begin{proposition} For odd integers $n$ and $m$, and any integer $k$ we have 
 \[
   \begin{aligned}
(F3)   \ \ \  \ \ \ \ \ \ \     &U_{k+1}C_{k}-C_{k+1}U_{k}=2,
     \ \ W_nV_{m-2}-V_nW_{m-2}=2\\
     (F4)   \ \ \  \ \ \ \ \ \ \    & U_{n} =V_{n}W_{n},
     \ \ U_{2k} = C_kU_k,\\
     (F5)   \ \ \  \ \ \ \ \ \ \     &C_{n}-2 = (q-2)W_{n}^2, \ \
     C_{n}+2 = (q+2)V_{n}^2,\\
     (F6)   \ \ \  \ \ \ \ \ \ \     &W_{nm} =W_{n}W_{m}(C_{n}), \ \
     V_{nm} =V_{m}V_{n}(C_{m}).
   \end{aligned}
 \]
 \end{proposition}
\begin{proof} Taking determinants of the matrices in  (2) and (3)
  we get (F3). Using (2)  we obtain the following equality of matrices. 
\[
\begin{aligned}
    & \left[ \begin{array}{cc} C_{2k} &       U_{2k} \\
      C_{2k+1}  &   U_{2k+1} \end{array} \right] = 
  \left[ \begin{array}{cc} 0
      & 1 \\ -1  & T \end{array} \right]^{k}
\left[ \begin{array}{cc} 0
      & 1 \\ -1  & T \end{array} \right]^{k}
  \left[ \begin{array}{cc} 2& 0 \\ T  & 1 \end{array} \right]=
  \\
  &\left[ \begin{array}{cc} -U_{k-1} &       U_{k} \\
      -U_{k}  &   U_{k+1} \end{array} \right]
  \left[ \begin{array}{cc} C_{k} &       U_{k} \\
      C_{k+1}  &   U_{k+1} \end{array} \right]
  =\left[ \begin{array}{cc} C_{k+1}U_{k}-C_kU_{k-1} & C_kU_{k}
      \\ C_{k+1}U_{k+1}-C_kU_{k}  & U_{k+1}^2-U_{k}^2 \end{array} \right].
\end{aligned}
  \]
  Comparing the second columns of the first and the last matrix
we get (F4).

Using $qU_{k}=U_{k-1}+U_{k+1}$ we obtain 
\[
\begin{aligned}
&C_{k+1}-C_k = U_{k+2}-U_{k}-U_{k+1}+U_{k-1} =
  \\
  &qU_{k+1}-2U_{k}-2U_{k+1}+qU_{k} =(q-2)W_{2k+1}.
\end{aligned}
\]
Comparing  the second rows  we get 
$ C_{k+1}U_{k+1}-C_kU_{k}=C_{2k+1}$, which in view of (F3) leads to 
\[
\begin{aligned}
  &(q-2)W_{2k+1}^2 = (C_{k+1}-C_k)(U_{k+1}+U_{k}) =\\
  &\left(C_{k+1}U_{k+1}-C_kU_{k}\right) -
  \left(C_kU_{k+1}-C_{k+1}U_{k}\right) = C_{2k+1}-2.
\end{aligned}
\]
To get the second part of (F5)
we substitute $-q$ for $q$ in the first part, and recall that
$W_{n}(-q) = (-1)^{\frac{n-1}{2}}V_{n}(q)$.

Now we use (F5) three times. 
\[
 \begin{aligned}
   &(q+2)V_{nm}^2(q) =C_{nm}+2= C_{n}(C_{m})+2=\\
   &\left(C_{m}(q)+2\right)V_{n}^2(C_{m})
   =(q+2)V_{m}^2(q)V_{n}^2(C_{m}).\\
\end{aligned}
\]
This gives us the second part of (F6)
up to a sign, which is easily established because all the polynomials
are monic.

Again by substituting $-q$ for $q$
we obtain the first part of (F6)  from the second part.
\end{proof}

\begin{proposition} 
For any odd integer $n$ we have 
\[\begin{aligned}
&(F7)   \ \ \  \ \ \ \ \ \ \       C_{n}(q)=qV_{n}(q^2-2),
\ \ \ U_{n}(q) = W_{n}(q^2-2),\\
&(F8)   \ \ \  \ \ \ \ \ \sqrt{4-q^2}U_{n}(q)=
(-1)^{\frac{n-1}{2}}C_{n}(\sqrt{4-q^2}),\\        &C_{n}(\sqrt{2+q})=
(-1)^{\frac{n-1}{2}}
\sqrt{2+q}U_{n}(\sqrt{2-q}).
\end{aligned}
\]
\end{proposition}
\begin {proof}
  For any integer $k$ using (2) and (1) we have 
\[
  \begin{aligned}
    & \left[ \begin{array}{cc} -U_{2k-1} &       U_{2k} \\
      -U_{2k}  &   U_{2k+1} \end{array} \right] = 
  \left[ \begin{array}{cc} 0
      & 1 \\ -1  & q \end{array} \right]^{2k}=
  \left[ \begin{array}{cc} -1
    & q \\ -q  & q^2-1 \end{array} \right]^{k}=
\\
&U_{k}(q^2-2)
  \left[ \begin{array}{cc} -1
    & q \\ -q  & q^2-1 \end{array} \right]-
  U_{k-1}(q^2-2)I.
\end{aligned}
\] 
Comparing the first rows of the matrices we get the following identities.
\[
U_{2k-1}(q) = U_{k}(q^2-2) +U_{k-1}(q^2-2)= W_{2k-1}(q^2-2),\ \
U_{2k}(q) = qU_{k}(q^2-2).
\]
The second part of (F7) is established, and we proceed to
\[
  qV_{2k-1}(q^2-2) =
  qU_{k}(q^2-2) - qU_{k-1}(q^2-2)
    =  U_{2k}(q)-U_{2k-2}(q) =
  C_{2k-1}(q)
\]
which gives us the first part of (F7).

The formula (F8) is a direct consequence of (F7).
\end{proof}
The formula (F9) can be obtained for odd $k$ from (F5)
by multiplying the two equations, and using (F4). Actually
(F5) can be understood as a finer resolution of (F9).

To get (F9) in full generality we set $k=2^sn$ with odd $n$.
We already have
\[
C_n^2(C_{2^s})+(4-C_{2^s}^2)U_n^2(C_{2^s})=4, \ \ \ C_{2^sn}^2+\frac{4-C_{2^s}^2}
{U_{2^s}^2} U_{2^sn}^2=4. 
\]
It remains to show that the fraction is equal to $4-q^2$, which
is just (F9) for $k=2^s$. It can be obtained directly by induction on
$s$, using $C_{2^{(s+1)}}=C_2(C_{2^s}) =C_{2^s}^2-2$ and $U_{2^{(s+1)}}=U_{2^s}C_{2^s}$.

The first 14 Chebyshev polynomials are listed. It transpires that
$C_p$ for a prime $p$ has all coefficients divisible by $p$,
with the exception of the leading coefficient $1$.
\[
\begin{aligned}
 &C_0 = 2, \ \ \ \ \  U_{14} =q^{13}-12q^{11}+55q^9-120q^7+126q^5-56q^3+7q\\    
 &C_1 = q, \ \ \ \ \ \ \ U_{13}=q^{12}-11q^{10}+45q^8-84q^6+70q^4-21q^2 +1\\
 &C_2 = q^2-2, \ \ \ \ \ \ \ \ \ U_{12} =q^{11}-10 q^9+36q^7-56q^5+35q^3-6q\\
 &C_3 = q^3-3q, \ \ \ \ \ \ \ \ \  U_{11}=q^{10}-9q^8+28q^6-35q^4+15q^2 -1\\
 &C_4 = q^4-4q^2 +2, \ \  \ \ \ \ \ \ \ \ \ U_{10}=q^9-8q^7+21q^5-20q^3+5q \\
 &C_5 = q^5-5q^3 +5q, \ \ \ \ \ \ \ \ \ \ \ \ \ U_9=q^8-7q^6+15q^4-10q^2 +1\\
 &C_6 = q^6-6q^4+9q^2 -2, \ \ \ \ \ \ \ \ \ \ \ \ \ U_8 = q^7-6q^5+10q^3 -4q\\
 &C_7 = q^7- 7q^5 +14q^3 -7q, \ \ \ \ \ \ \ \ \ \ \ \ \ U_7=q^6-5q^4+6q^2 -1\\ 
 &C_8 = q^8-8q^6+20q^4-16q^2 +2,\ \ \ \ \ \ \ \ \ \ \ \ \ \ U_6 = q^5-4q^3 +3q\\
  &C_9 = q^9-9q^7+27q^5 -30q^3 +9q, \ \ \ \ \ \ \ \ \ \ \ \ \ \ \
  U_5 = q^4-3q^2 +1\\
  &C_{10} = q^{10}-10q^8+35q^6-50q^4+25q^2 -2, \ \ \ \ \ \ \ \ \ \ \ \
  U_4 = q^3-2q\\
  &C_{11} = q^{11}-11q^9+44q^7-77q^5 +55q^3 -11q, \ \ \ \ \ \ \ \ \ \ \
  U_3 = q^2-1\\
 &C_{12} = q^{12}-12q^{10}+54q^8-112q^6+105q^4-36q^2 +2, \ \ \ \ \ \ \ \ \ U_2 = q\\
 &C_{13} = q^{13}-13q^{11}+65q^9-156q^7+182q^5 -91q^3 +13q,\ \ \ \ \ \ \ \ \ U_1 = 1
\end{aligned}
\]

\end{document}